\documentclass[letterpaper,11pt]{article}

\usepackage{amsmath,amsfonts,amssymb,amsthm,mathrsfs,dsfont}
\usepackage{cases}
\usepackage{enumerate}
\usepackage[usenames,dvipsnames]{color}
\usepackage{verbatim} % for \begin{comment} ... \end{comment}
\usepackage{array} % for tables
\usepackage{graphicx,epstopdf} % for inserting figures
\usepackage{mathtools}
\theoremstyle{plain} \newtheorem{thm}{Theorem}[section]
\theoremstyle{plain} \newtheorem{lemma}[thm]{Lemma}
\theoremstyle{plain} \newtheorem{prop}[thm]{Proposition}
\theoremstyle{plain} \newtheorem{cor}[thm]{Corollary}
\theoremstyle{plain} 
\theoremstyle{plain} \newtheorem{question}[thm]{Question}
\theoremstyle{definition} 
\theoremstyle{definition} 
\theoremstyle{plain} 

\newcommand{\red}[1]{\textcolor{Red}{#1}}

\newcommand{\sub}[0]{\subseteq}
\newcommand{\sm}[0]{\setminus}
\renewcommand{\dots}[0]{,\ldots,}
\newcommand{\ra}[0]{\rightarrow}

\newcommand{\ov}[0]{\overline}

\newcommand{\beq}[1]{\begin{equation}\label{#1}}
\newcommand{\enq}[0]{\end{equation}}

\newcommand{\bn}[0]{\bigskip\noindent}
\newcommand{\mn}[0]{\medskip\noindent}
\newcommand{\nin}[0]{\noindent}

\newcommand{\0}[0]{\varnothing}
\newcommand{\E}[0]{{\mathbb{E}}}

\newcommand{\Cc}[0]{\tbinom}
\newcommand{\ga}[0]{\alpha }
\newcommand{\gb}[0]{\beta }
\newcommand{\gc}[0]{\gamma }
\newcommand{\gd}[0]{\delta }
\newcommand{\gD}[0]{\Delta }
\newcommand{\gG}[0]{\Gamma }

\newcommand{\gl}[0]{\lambda }
\newcommand{\gL}[0]{\Lambda}
\newcommand{\go}[0]{\omega}
\newcommand{\gO}[0]{\Omega}

\newcommand{\gs}[0]{\sigma}

\newcommand{\gz}[0]{\zeta}
\newcommand{\eps}[0]{\varepsilon }
\newcommand{\vt}[0]{\vartheta}
\newcommand{\vs}[0]{\varsigma}

\newcommand{\cee}[0]{{\cal C}}
\newcommand{\D}[0]{{\cal D}}
\newcommand{\eee}[0]{{\cal E}}

\newcommand{\pee}[0]{{\cal P}}
\newcommand{\Q}[0]{{\cal Q}}
\newcommand{\R}[0]{{\cal R}}
\newcommand{\sss}[0]{{\cal S}}
\newcommand{\T}[0]{{\cal T}}
\newcommand{\U}[0]{{\cal U}}

\newcommand{\W}[0]{{\cal W}}

\newcommand{\q}[0]{q}

\author{Jacob D. Baron\thanks{Department of Mathematics, Rutgers University, Piscataway, NJ. 
Supported by the U.S. Department of Homeland Security under Grant Award 2012-ST-104-000044. The views and conclusions contained in this document are those of the authors and should not be interpreted as necessarily representing the official policies, either express or implied, of the U.S. Department of Homeland Security.} \and Jeff Kahn\thanks{Department of Mathematics, Rutgers University, Piscataway, NJ. Supported by the National Science Foundation under Grant Awards DMS1201337 and DMS1501962.} }

\title{On the Cycle Space of a Random Graph}

\date{Sept 2016}

\begin{document}

\maketitle

\begin{abstract}
Write $\cee(G)$ for the cycle space of a graph $G$, $\cee_\kappa(G)$ for the subspace of $\cee(G)$ spanned by the copies of the $\kappa$-cycle $C_\kappa$ in $G$, $\T_\kappa$ for the class of graphs satisfying $\cee_\kappa(G)=\cee(G)$, and $\Q_\kappa$ for the class of graphs each of whose edges lies in a $C_\kappa$. We prove that for every odd $\kappa \geq 3$ and $G=G_{n,p}$, \[\max_p \, \Pr(G \in \Q_\kappa \setminus \T_\kappa) \rightarrow 0;\] so the $C_\kappa$'s of a random graph span its cycle space as soon as they cover its edges. For $\kappa=3$ this was shown in \cite{DHK}.
\end{abstract}

\section{Introduction}\label{Intro}

An issue of considerable interest in combinatorics over the last few decades
has been the extent to which
various standard facts,
for instance the classic theorems of Tur\'an, Ramsey and Szemer\'edi, remain true
in a ``sparse random'' setting. Thus, for example,
one may ask for which $p$ a given (deterministic) assertion regarding the complete graph $K_n$
is likely to hold in the (``Bernoulli") random graph $G_{n,p}$. Our main result
follows this theme.

\medskip
Our underlying deterministic statement is Proposition~\ref{prop:C5Kn}
below, for which we need a few definitions. The \emph{edge space} of a graph $G$,
denoted $\mathcal{E}(G)$, is the vector space $\mathbb{F}_2^{E(G)}$.
Its elements are naturally identified with the (spanning) subgraphs of $G$. The \emph{cycle space}
of $G$, denoted $\mathcal{C}(G)$, is the subspace of $\mathcal{E}(G)$ generated by the
(indicators of) cycles of $G$ (see e.g. \cite[Sec.\ 1.9]{Diestel} for an exposition).
For a  fixed graph $H$, the $H$-{\em space} of $G$ is the subspace of
$\mathcal{E}(G)$ generated by the copies of $H$ in $G$; this will be denoted
$\cee_H(G)$, or simply $\cee_\kappa(G)$ if $H=C_\kappa$ (the $\kappa$-cycle or $\kappa$-\emph{gon}).

\begin{prop}
\label{prop:C5Kn}
If $\kappa \geq 3$ is odd, then for any $n \geq \kappa$, $\cee_\kappa(K_n) = \mathcal{C}(K_n)$.
\end{prop}

\nin
Of course for even $\kappa$, $\cee_\kappa(G)$ is at most the space spanned by \emph{even} cycles.
Below, in Theorem \ref{thm:Hspace}, we will characterize $\cee_H(K_n)$ for any
fixed $H$ and large enough $n$.

Assuming $\kappa$ is odd,
when, in terms of $p$ ($=p(n)$), are the $\kappa$-gons of $G_{n,p}$ likely to span its cycle space?
Let $\mathcal{T}_\kappa$
be the class of graphs $G$ satisfying $\cee_{\kappa}(G) = \mathcal{C}(G)$
and let $\mathcal{Q}_\kappa$ be the class of nonempty graphs each of whose edges
lies in a copy of $C_\kappa$. For any $G$,
it's easy to see that $G \notin \mathcal{T}_\kappa$ unless every edge of $G$ that lies in a cycle
in fact lies in a $\kappa$-gon. On the other hand, if $p> (1+\gO(1))\log n/n$ then
w.h.p.\footnote{W.h.p. (``with high probability'') means with probability tending to 1 as
$n \rightarrow \infty$.}
every edge of $G_{n,p}$
does lie in a cycle (e.g. \cite[p.\ 105]{JLR}).  So for such $p$, $G_{n,p} \in \mathcal{T}_\kappa$
w.h.p. at least requires $G_{n,p} \in \Q_\kappa$ w.h.p., and
we should first understand when this is true.
Let
\begin{align}
\label{eq:p^*intro}
p^*_\kappa = p^*_\kappa(n) = [(\kappa/(\kappa-1))n^{-(\kappa-2)}\log n]^{1/(\kappa-1)}
\end{align}
(where we always use $\log$ for $\ln$).
Note $\Q_\kappa$ is not an \emph{increasing} property---that is, it is not preserved by adding edges.
Nonetheless, $p^*_\kappa$ is a \emph{sharp threshold} for $\mathcal{Q}_\kappa$,
in the sense that:

\begin{lemma}
\label{lem:Qsharpthresh}
For any fixed $\kappa \geq 3$ and $\eps>0$,
\begin{align}
\label{eq:lemQsharpthresh}
\Pr(G_{n,p} \in \mathcal{Q}_{\kappa}) \rightarrow
\begin{cases}
0 &\text{ if } p < (1-\eps)p^*_\kappa, \\
1 &\text{ if } p > (1+\eps)p^*_\kappa.
\end{cases}
\end{align}
\end{lemma}
\noindent
(Throughout the paper limits are taken as $n\ra\infty$.)
We prove this routine observation in Section \ref{SEC:LemmaThresh}.
The cases in \eqref{eq:lemQsharpthresh} are called the \emph{0-statement}
and the \emph{1-statement} (respectively).

Given Lemma \ref{lem:Qsharpthresh}, one might hope that $p^*_\kappa$ is also a sharp
threshold for $\mathcal{T}_\kappa$, and it essentially is, but for a small glitch in the
0-statement: for $p< (1-\gO(1))/n$,
we have $\lim\Pr(G_{n,p} \in \mathcal{T}_\kappa)>0$
for
the silly reason that the probability of having no cycles at all
is (asymptotically) positive (see e.g. \cite[Thm. 1]{Pittel}). Thus we will show:

\begin{thm}
\label{thm:C5threshold}
For any fixed odd $\kappa \geq 3$ and $\eps>0$,
\begin{align*}
\label{eq:lemQsharpthresh}
\Pr(G_{n,p} \in \mathcal{T}_{\kappa}) \rightarrow
\begin{cases}
0 &\text{ if } (1-o(1))/n < p < (1-\eps)p^*_\kappa, \\
1 &\text{ if } p > (1+\eps)p^*_\kappa.
\end{cases}
\end{align*}
\end{thm}

We actually prove the following stronger statement
(see Section~\ref{SEC:LemmaThresh} for ``stronger"), which says that edges not in
$\kappa$-gons are the obstruction to $\T_\kappa$ in a precise sense. This is our main result.

\begin{thm}
\label{thm:C5precise}
For any fixed odd $\kappa \geq 3$,
\beq{orig}
 \max_p \, \Pr(G_{n,p} \in \mathcal{Q}_\kappa \setminus \T_\kappa) \ra 0;
 \enq
equivalently,
\beq{equiv}
\forall \, p=p(n), ~~~\Pr(G_{n,p} \in \mathcal{Q}_\kappa \setminus \T_\kappa) \ra 0.
\enq
\end{thm}
\nin
(The (trivial) equivalence is given by the observation that
\eqref{equiv} holds iff it holds when, for each $n$, $p=p(n)$ is a value
achieving the maximum in \eqref{orig} (and in this case the two statements
are the same).)

Theorems~\ref{thm:C5threshold} and
\ref{thm:C5precise} for $\kappa=3$ were proved in \cite{DHK};
even the former had been open and of interest, being the first unsettled case of a
conjecture of M.\ Kahle (see
\cite{Kahle,KahleArxiv}) on the homology of the clique complex of
$G_{n,p}$. Perhaps surprisingly, the argument of \cite{DHK} does not extend
to $\kappa \geq 5$, though, as discussed below, it does share a starting point with what we do here.

What happens if we replace the $C_\kappa$ of Proposition \ref{prop:C5Kn} by some other graph?
With $\mathcal{D}(G) = \{D \in \mathcal{E}(G) : |D| \equiv 0\pmod{2}\}$,
the proposition generalizes neatly:

\begin{thm}
\label{thm:Hspace}
For any graph $H$ with at least one edge and n large enough with respect to H,
\begin{align}
\label{eq:Hspace}
\cee_H(K_n) =
\begin{cases}
\mathcal{C}(K_n) &\text{ if } H \text{ is Eulerian and } |H| \text{ is odd,} \\
\mathcal{C}(K_n) \cap \mathcal{D}(K_n) &\text{ if } H \text{ is Eulerian and } |H| \text{ is even,} \\
\mathcal{E}(K_n) &\text{ if } H \text{ is not Eulerian and } |H| \text{ is odd,} \\
\mathcal{D}(K_n) &\text{ if } H \text{ is not Eulerian and } |H| \text{ is even.}
\end{cases}
\end{align}
\end{thm}
\noindent Here $|H| = |E(H)|$ and ``Eulerian'' means degrees are even, but not that
the graph is necessarily connected.
Of course the left-to-right containments ($\cee_H(K_n)\sub \cee(K_n)$ and so on) are obvious.
%(If $|H|$ is even (resp. Eulerian), then so is each element of $\cee_H(K_n)$.)
%The easy proof of Theorem~\ref{thm:Hspace} is given in Section~\ref{SEC:ThmHspace}.

The natural value of $\cee_H(G)$, which we will denote $\W_H(G)$,
is then what one gets by replacing
$K_n$ by $G$ in the appropriate expression on the r.h.s. of \eqref{eq:Hspace};
e.g. for $H=C_\kappa$,
\beq{W}
\mathcal{W}_H(G)=\left\{\begin{array}{ll}
\cee(G)&\mbox{if $\kappa$ is odd,}\\
\cee(G)\cap \D(G)&\mbox{if $\kappa$ is even.}
\end{array}\right.
\enq
(We could instead set
$\mathcal{W}_H(G) = \eee(G)\cap \cee_H(K_n)$, which by
Theorem~\ref{thm:Hspace} is the same for all but a few
values of $n$.)
So we are interested in understanding when $G_{n,p}$ is likely to lie in
\[
\mathcal{T}_H := \{G:\cee_{H}(G) = \mathcal{W}_H(G)\}.
\]
(Again, $\cee_H(G) \subseteq \mathcal{W}_H(G)$ is trivial for any $H$ and $G$.)

As before, membership in $\T_H$ will (in non-silly cases) at least require that the copies of $H$
cover the edges of $G:=G_{n,p}$, but when $H$ is non-Eulerian there is a second
requirement:
each vertex of $G$ should have odd degree in some copy of $H$ in $G$
(since for any $v\in V(G) $, $\W_H(G)$ will contain graphs in which $v$ has odd degree).
For example if $H$ is a pair of triangles joined by a slightly long path and
$n^{-1+\eps}< p\ll n^{-2/3}$
for a suitable small $\eps$ depending on the length of the path,
then (w.h.p.) all edges of $G$ are in copies of $H$, but
most vertices fail to lie in triangles, so have even degree in every copy.

Generalizing $\mathcal{Q}_\kappa$, let $\mathcal{Q}_H$ be the class of nonempty
graphs $G$ satisfying (i) each edge of $G$ is in a copy of $H$, and
(ii) if $H$ is not Eulerian,
then each vertex of $G$ has odd degree in some copy of $H$;
so we have just said that we ``essentially" have
$\T_H\sub \Q_H$.
Though we hesitate to make it
a conjecture, we don't know that the following generalization of Theorem~\ref{thm:C5precise}
is wrong.
\begin{question}\label{Gen'l?}
Could it be that for each (fixed) H,
\beq{origH}
\max_p \, \Pr(G_{n,p} \in \mathcal{Q}_H \setminus \T_H) \ra 0?
\enq
\end{question}
\nin

\mn
Understanding when $G_{n,p}\in \Q_H$ w.h.p.\ is easier, so this would also tell us
when $\T_H$ is likely to hold.
(Note that in general we don't expect a statement like Theorem~\ref{thm:C5threshold},
since the ``threshold" for $\Q_H$ itself may not be sharp.)
Even if \eqref{origH} is not true in general, it seems likely
to hold for reasonably nice $H$ (even, say, edge-transitive to start, though this
should be much more than is needed).
One could also relax \eqref{origH} to an Erd\H{o}s-R\'enyi-like threshold statement;
e.g. with
$p_{\Q_H}=\min\{p_0: \Pr(G_{n,p} \in \mathcal{Q}_H )\geq 1/2 ~\forall p\geq p_0\}$,
%{\em viz.}
\[
\mbox{if $p\gg p_{\Q_H} $ then $G_{n,p}\in \T_H$ w.h.p.}
\]

\bn \textbf{\emph{Outline.}} The rest of the paper is organized as follows.
Usage notes conclude the introduction. Section~\ref{SEC2} recalls edge space
preliminaries
and outlines the main points (Lemmas~\ref{lem:p=O(p)'}-\ref{LemCG})
for the proof of Theorem~\ref{thm:C5precise}.
Section~\ref{SEC:Tools} reviews tools and
derives some relatively routine consequences.
%One more interesting point here is
%Theorem~\ref{thm:BKStochDom}, a natural extension of the well-known
%``BK Inequality" of van den Berg and Kesten \cite{BK} that,
%surprisingly, appears to be new.
%(The reason for including this will be clear when we get to it,
%but as we will make no use of it here, the proof will be given \red{elsewhere}.)
Section~\ref{SEC:LemmaThresh} proves Lemma~\ref{lem:Qsharpthresh}
and gives the easy derivation of Theorem~\ref{thm:C5threshold} from
Theorem~\ref{thm:C5precise}.

The heart of the paper is
Sections~\ref{SEC:PLCG}-\ref{SEC:LemmaP=O(p*)},
which
prove Lemmas~\ref{LemCG}, \ref{lem:couplingdown}
and \ref{lem:p=O(p)'}.  These are, respectively, very easy (modulo a big machine);
easy but a little circuitous; and not so easy and \emph{quite} circuitous
(and by far the most interesting part of the argument).
Finally, Section \ref{SEC:ThmHspace} gives the easy proof of Theorem~\ref{thm:Hspace},
which we postpone as it is unrelated to the rest of the paper.

\bn \textbf{\emph{Usage.}}  Given a graph $G$, we will use $V$ and $E$ for $V(G)$ and $E(G)$
when the meaning is clear.
We will often identify graphs with their edge sets.

For $v\in V$ and $F \subseteq G$ we use $N_F(v) = \{x : vx \in F\}$ and $d_F(v) = |N_F(v)|$.
For disjoint $A, B \subseteq V$, $\nabla_F(A,B)$ is the set of $F$-edges joining $A$ and $B$,
and we use $\nabla_F(A)$ for $\nabla_F(A, V\setminus A)$---these are the {\em cuts} of $G$---and
$\nabla_F(v)$ for $\nabla_F(\{v\})$.  In all cases we drop the subscripts when $F=G$.
As usual $\alpha(G)$ and $\Delta(G)$ (or $\gD_G$)
denote independence number and maximum degree of $G$.
We will sometimes use $v_G$ and $e_G$ for the numbers of vertices and edges of $G$.

We use $[n]$ for $\{1\dots n\}$ (for positive integer $n$),
$\log$ for $\ln$ and
$a = (1 \pm b)c$ for $(1-b)c \leq a \leq (1+b)c$.
Asymptotic notation ($\sim$, $O(\cdot)$, $\gO(\cdot)$ and so on), is standard,
with $a\ll b$ and $a\asymp b$ replacing $a=o(b)$ and $a=\Theta(b)$ when convenient.
Throughout the paper we assume $n$ is large enough to support our various assertions,
and usually pretend large numbers are integers.

\section{Main points for the Proof of Theorem~\ref{thm:C5precise}}
\label{SEC2}

Before outlining the proof of Theorem~\ref{thm:C5precise},
we need to review just a little more background.

\subsection{Edge space basics}
\label{subsec:EdgeSpace}
The edge space $\mathcal{E}(G)$ of a graph $G$ (defined in the paper's second paragraph),
being an $\mathbb{F}_2$-vector space, comes equipped with a standard inner product:
$\langle J,K\rangle = \sum_{e \in E(G)} J(e)K(e) = |J \cap K|$, where the sum and cardinality
are interpreted mod 2. (The first expression thinks of $J$ and $K$ as vectors, the second as
subgraphs of $G$.) With this, the orthogonal complement, $\mathcal{S}^\perp$, of a subspace
$\mathcal{S}$ of $\mathcal{E}(G)$ is defined as usual. Then $\mathcal{C}^\perp(G)$, called
the \emph{cut space} of $G$, consists of the (indicators of) cuts of $G$ (which, note,
includes $\0$);
$(\cee(G)\cap \D(G))^\perp$ consists of cuts and their complements; and
$\cee_H^\perp(G)$ is the set of
subgraphs of $G$ having even intersection with every copy of $H$ (in $G$).

As mentioned earlier, $\cee_H(G) \subseteq \mathcal{W}_H(G)$ always; dually,
$\mathcal{W}^\perp_H(G) \subseteq \cee^\perp_H(G)$. In particular, for odd $\kappa \geq 3$,
\begin{align}
\label{eq:edgespaceperps}
\mathcal{C}^\perp(G) \subseteq \cee^\perp_\kappa(G), \text{ and equality here is the
same as $G \in \T_\kappa$.}
\end{align}

The next (trivial) observation will be useful at a few points.

\begin{prop}\label{prop:addcut}
Let $G$ be a graph and $L\sub G$, and suppose
$L',L''$ are (respectively) smallest and largest members of
the coset $L + \cee^\perp(G)$.  Then
\[
\forall \, v\in V
~~~ d_{L'}(v)\leq d_{G}(v)/2 \leq
d_{L''}(v).
\]
\end{prop}
\nin
(For example if $d_{L'}(v)>d_{G}(v)/2$, then $L' + \nabla(v)$
($\in L + \cee^\perp(G)$)
is smaller than $L'$.)

In particular, if $G \notin \T_\kappa$, then since
$\cee_{\kappa}^\perp(G) \setminus \mathcal{C}^\perp(G) \supseteq L+\cee^\perp(G)$
for any $L \in \cee_{\kappa}^\perp(G) \setminus \mathcal{C}^\perp(G)$, a smallest
element $F$ of $\cee_{\kappa}^\perp(G) \setminus \mathcal{C}^\perp(G)$ satisfies
\begin{align}
\label{dFdG}
d_F(v)\leq d_{G}(v)/2 \;\;\; \forall \, v \in V.
\end{align}

\subsection{Structure of the proof}
\label{SEC:Overview}

For the rest of the paper we fix an odd $\kappa \geq 5$ (as mentioned earlier,
the case $\kappa=3$ of
Theorem~\ref{thm:C5precise} was proved in \cite{DHK}),
and set
$p^*=p^*_\kappa$, $\Q=\Q_\kappa$ and $\T=\T_\kappa$; so our objective,
\eqref{orig}, becomes
\beq{orig'}
 \max_p \, \Pr(G_{n,p} \in \Q\sm \T) \ra 0.
 \enq

As sometimes happens, though \eqref{orig'} should become ``more true'' as $p$ ($>p^*$) grows,
some points in the proof
run into difficulties for larger $p$, and it seems easiest to deal first with smaller $p$
and then derive the full statement from this restricted version. The next two lemmas,
the first of
which is our main point, implement this plan.

\begin{lemma}
\label{lem:p=O(p)'}
For any fixed $K$ and $p\leq  Kp^*$,
\beq{MP}
\Pr(G_{n,p} \in \Q\setminus \T)\ra 0.
\enq
\end{lemma}
\nin
(The interest here is really in $p$ at least about $p^*$, smaller
values being handled by Lemma~\ref{lem:Qsharpthresh}; see
\eqref{eq:pgeq1-eps}.)

\begin{lemma}
\label{lem:couplingdown}
There exists $K$ such that if $p>q:=Kp^*$, then
\[
\Pr(G_{n,p}\notin \T) < \Pr(G_{n,q}\notin \T) + o(1).
\]
\end{lemma}

\nin
Applying Lemmas~\ref{lem:couplingdown} and \ref{lem:p=O(p)'}, together with (the
1-statement of) Lemma~\ref{lem:Qsharpthresh} to $p'(n):=\min\{p(n),Kp^*(n)\}$
then easily gives
Theorem~\ref{thm:C5precise}.
(For $n$'s with $p(n)>Kp^*$, we have,
using Lemma~\ref{lem:couplingdown} for the first inequality and
Lemmas~\ref{lem:p=O(p)'} and
\ref{lem:Qsharpthresh} for the final $o(1)$,
\begin{align*}
\Pr(G_{n,p}\in \Q\sm \T)&< \Pr(G_{n,p'}\not\in \T) +o(1)\\
&< \Pr(G_{n,p'}\in \Q\sm \T) +\Pr(G_{n,p'}\not\in \Q)+o(1) =o(1),
\end{align*}
and for the remaining $n$'s we have $p=p'$ and
Lemma~\ref{lem:p=O(p)'} applies directly.)

%%%%%%%%%%%%%%%%%%%%%%%
\iffalse
these are enough: if $p=p(n)$ violates \eqref{equiv},
then Lemma~\ref{lem:couplingdown}, with the 1-statement of Lemma~\ref{lem:Qsharpthresh}, says
$p'(n):=\min\{p(n),Kp^*(n)\}$ (with $K$ as in Lemma \ref{lem:couplingdown})
violates Lemma~\ref{lem:p=O(p)'}.
\fi
%%%%%%%%%%%%%%%%%%%%%%%

\medskip
The following device will play a central role in the proofs of both of these lemmas
(so in most of the paper).
For the remainder of our discussion we fix some rule that associates with each finite
graph $G$ a subgraph $F(G)$ satisfying
\begin{align}
\label{eq:f}
F(G) =
\begin{cases}
\0 &\text{if } G \in \T, \\
\text{some smallest element of } \cee_{\kappa}^\perp(G)
\setminus \cee^\perp(G) &\text{if } G \notin \T.
\end{cases}
\end{align}
%(Note we always have
%$\cee^\perp(G) \subseteq \cee^\perp_\kappa(G)$---the dual of
%$\cee_\kappa(G) \subseteq \cee(G)$---with equality iff $G\in \T_\kappa$,
%so the definition makes sense.)

\nin
We will use this only with $G=G_{n,p}$, so set $F(G_{n,p})=F$ throughout.
A crucial point is that $G$ determines $F$ (for ``crucial" see the paragraph preceding
Proposition~\ref{prop:Frightsize}).
That $F$ is a
minimizer will be used only to say that it is small and has small
degrees, as promised by \eqref{dFdG}.

With $F$ thus defined we may replace the event $\{G_{n,p}\notin \T\} $
by the more convenient $\{F \neq \0 \}$, which in particular allows us
to tailor our treatment to the size of a hypothetical $F$.
As we will see, ruling out fairly large $F$'s is easy---not
from scratch, but with the help of a powerful result
from \cite{Conlon-Gowers} (Theorem~\ref{thm:ConlonGowers} below), which more
or less immediately yields:

\begin{lemma}
\label{LemCG}
For fixed $c>0$ and $p\gg n^{-(\kappa-2)/(\kappa-1)}$,
\beq{eq:LemCG}
\Pr(|F|>c n^2p)\ra 0.
\enq
\end{lemma}
\nin
Thus the real problem in proving Lemma~\ref{lem:p=O(p)'},
and the most interesting part of the whole business,
is dealing with $F$'s that are small relative to
$G$ (but nonempty).
%---that is, proving that for $p \leq Kp^*$ and $\lambda = \lambda(n) \ra 0$,
%\[\Pr((G_{n,p} \in \Q) \; \wedge \; (0 < |F| < \lambda n^2p)) \ra 0.\]
Thus far---and a little further;
see the preview following the statement of Lemma~\ref{lem:InterruptedPathsBound}---our
structure mirrors that of
\cite{DHK}, but the (two-page) argument handling this main point there
offers no help here.

\mn
\emph{Remark.}
In connection with Question~\ref{Gen'l?},
it seems worth noting here that Lemma~\ref{LemCG}, at least, extends
to considerably more general $H$;
see
Section~\ref{SEC:PLCG} for a little more on this.

\section{Tools}
\label{SEC:Tools}

\subsection{Deviation and correlation}
\label{Deviation and correlation}
Set
\beq{eq:varphidef}
\varphi(x) = (1+x)\log(1+x)-x
\enq
for $ x > -1$ and (for continuity) $ \varphi(-1)=1$.
We use ``Chernoff's Inequality" in the following form; see for example
\cite[Thm. 2.1]{JLR}.
\begin{thm}
\label{thm:Chernoff}
If $X \sim \mathrm{Bin}(n,p)$ and $\mu = \mathbb{E}[X] = np$, then for $t \geq 0$,
\begin{align}
\label{eq:ChernoffUpper}
\Pr(X \geq \mu + t) &\leq \exp\left[-\mu\varphi(t/\mu)\right]
\leq \exp\left[-t^2/(2(\mu+t/3))\right], \\
\label{eq:ChernoffLower}
\Pr(X \leq \mu - t) &\leq \exp[-\mu\varphi(-t/\mu)]
\leq \exp[-t^2/(2\mu)].
\end{align}
\end{thm}

\nin
For larger deviations the following consequence of the finer bound in \eqref{eq:ChernoffUpper}
will be convenient.
\begin{thm}
\label{thm:Chernoff'}
For $X\sim B(n,p)$ and any $K$, with $\mu=\mathbb{E}[X]=np$,
\begin{eqnarray*}
\Pr(X > K\mu) < \exp[-K\mu \log (K/e)].
\end{eqnarray*}
\end{thm}

\nin (Of course this is only helpful if $K>e$.)

\medskip
We will make substantial use of
the following fundamental lower tail bound of
Svante Janson (\cite{Janson} or
\cite[Theorem 2.14]{JLR}), for which
we need a little notation.
Suppose $A_1\dots A_m$ are subsets of the
finite set $\gG$. Let $\Gamma_p$ be the random subset of $\Gamma$
gotten by including each $x$ ($ \in \Gamma$) with probability $p$,
these choices made independently.
For $j\in [m]$, let $I_j$ be the indicator of the event
$\{\gG_p\supseteq A_j\}$, and set $X=\sum I_j$,
$\mu = \mathbb{E} X =\sum_j\mathbb{E} I_j$
and
\beq{Delta}
\mbox{$\ov{\gD} = \sum\sum\{\mathbb{E} I_iI_j: A_i\cap A_j\neq\0  \}.$}
\enq
(Note this includes the diagonal terms.)

\begin{thm}\label{TJanson}
With notation as above, for any $t\in [0,\mu]$,
\[\Pr(X\leq \mu -t) \leq \exp[-\varphi(-t/\mu)\mu^2/\ov{\gD}] \leq \exp[-t^2/(2\ov{\gD})].\]
\end{thm}

The next result is \cite[Lemma 2.46]{JLR} (originally \cite[Lemma 2]{Janson}).
\begin{lemma}\label{JUB}
For events $A_1\dots A_n$
in a probability space, and $\mu=\sum\Pr(A_i)$,
\begin{align*}
\Pr(\mbox{some $\mu+t$
independent $A_i$'s occur})
&\leq
\exp\left[-\mu\varphi(t/\mu)\right]\\
&\leq \exp\left[-t^2/(2(\mu+t/3))\right].
\end{align*}
\end{lemma}
\nin
Note the bound here is the same as the one in \eqref{eq:ChernoffUpper}, which is
thus contained in Lemma~\ref{JUB}. 
(Strictly speaking, \cite{Janson} and \cite{JLR} state Lemma~\ref{JUB} only in setting
of Theorem~\ref{TJanson}, but the proofs there are valid for the version here.) 
Lemma \ref{JUB} implies the weaker but sometimes convenient
\beq{ET}
\Pr(\mbox{some $l$
independent $A_i$'s occur})~\leq ~\mu^l/l! ~\leq ~(e\mu/l)^l
\enq
observed in
\cite{Erdos-Tetali} (or see \cite[Lemma 8.4.1]{AS}).
%See following Theorem~\ref{thm:BKStochDom} for a little more on
%Lemma~\ref{JUB}.

\mn

The setting for the next %two theorems
theorem is a finite product probability space
$\Omega=\prod_{i=1}^t \Omega_i$ with each factor linearly ordered.
As usual an event $A \subseteq \Omega$ is \emph{increasing} if its indicator is a
nondecreasing function (with respect to the product order on $\Omega$) and \emph{decreasing} if
its complement is increasing.
The seminal ``correlation inequality" is essentially due to Harris \cite{Harris}:

\begin{thm}
\label{thm:Harris}
If $A,B\sub \gO$ are either both increasing or both decreasing, then
\[
\Pr(A \cap B) \geq \Pr(A)\Pr(B);\]
if one is increasing and the other decreasing then the inequality is reversed.
\end{thm}

\subsection{Density generics}

{\em From now on we use
$G$ for $G_{n,p}$ and $V$ for $[n]=V(G)$.}
%
%Throughout this section, $G=G_{n,q}$ and $V=V(G) $ ($=[n]$).
%
Theorems~\ref{thm:Chernoff} and \ref{thm:Chernoff'}
easily imply the next two standardish propositions, whose proofs we omit.

\begin{prop}
\label{prop:routine}
For $p \gg n^{-1}\log n$, w.h.p.
\[
\mbox{$|G| \sim n^2p/2~~$ and $~~d(v) \sim np \;\, \forall \ v \in V$.}
\]
\end{prop}
\nin
(Of course the second conclusion implies the first, which just needs $p\gg n^{-2}$.)

\begin{prop}\label{density}
{\rm (a)}
For any $\eps>0$ there is a $K$
such that w.h.p. for all disjoint $S,T\sub V$ with
$|S|,|T|>Kp^{-1}\log n$
\[
\mbox{$|\nabla_G(S,T)| =(1\pm \eps) |S||T|p$}
% ~ $\forall$ disjoint $S,T\sub V$ with $|S|,|T|>Kp^{-1}\log n$}
\]
and
\[
\mbox{$|G[S]| =(1\pm \eps) \Cc{|S|}{2}p$.}
%~ $\forall$ $S\sub V$ with $|S|>Kp^{-1}\log n$}
\]
{\rm (b)}
For $K>3$ w.h.p.
\[
\mbox{$|G[S]|<  K|S|\log n$ for all $S\sub V$ with
$|S|\leq Kp^{-1}\log n$.}
\]
{\rm (c)}
For each $\eps>0$ there is a $K$ such that if $p>Kn^{-1}\log n$ then w.h.p.
\[
\mbox{$|\nabla_G(S)|=(1\pm \eps)|S|(n-|S|) ~~\forall S\sub V$.}
\]
\end{prop}
\nin
%(It's not hard to derive (c) from (a) and (b).)

\begin{prop}
\label{cpts}
For fixed $\eps >0$ and $p \gg 1/n$, w.h.p.:
if $H \sub G$ satisfies
\beq{deltaF}
d_H(v) > (1-\eps)np/2 ~~~\forall \; v \in V,
\enq
then no component of $H$ has size less than $(1-2\eps)n/2$.
\end{prop}

\begin{proof}
For a given $W\sub V$ of size $w< (1-2\eps)n/2$, let $\chi =|G[W]|$.
Then $\mu:=\mathbb{E}\chi = \binom{w}{2}p< w^2p/2$, while if $W$ is a component of an $H$
satisfying \eqref{deltaF} then
\[
\chi \geq |H[W]| > w(1-\eps)np/4 > \tfrac{(1-\eps)n}{2w}\mu =:K\mu.
\]
But (since $K>(1-\eps)/(1-2\eps)=1+\gO(1)$)
Theorems~\ref{thm:Chernoff} and \ref{thm:Chernoff'} give
\[
\gc_w:=\Pr(\chi >K\mu) <\left\{\begin{array}{ll}
\exp[-\gO(\mu)]&\mbox{if $K<e^2$ (say),}\\
\exp[-K\mu\log (K/e)]&\mbox{otherwise.}
\end{array}\right.
\]
Thus, with sums over $w\in (0,(1-2\eps)n/2)$,
the probability that some $H$ as in the lemma admits a component of size less
than $(1-2\eps)n/2$ is less than
\[
\sum\Cc{n}{w} \gc_w < \sum\exp[w\log(en/w)]\gc_w,
\]
which for $p\gg 1/n$ is easily seen to be $o(1)$.
\end{proof}

Finally, we need to know a little about the adjacency matrix, $A(G)$, of $G$.
A version of \eqref{eq:eigvals} below was proved in \cite{Furedi-Komlos1980}
(see also \cite{Alon-Kriv-Vu2002}) and \eqref{eq:eigvecv1} is shown (e.g.) in \cite{Mitra2009}.

\begin{prop}
\label{prop:spectrum}
Let $\lambda_1 \geq \lambda_2 \geq \ldots \geq \lambda_n$ be the eigenvalues of $A(G)$
and $v_1, v_2, \ldots, v_n$ associated orthonormal eigenvectors, say with
$\max_j v_{1,j} > 0$. If $p \gg n^{-1}\log n$, then w.h.p.
\begin{align}
\label{eq:eigvals}
\lambda_1 \sim np ~~~\text{ and }~~~ \max\{\lambda_2, |\lambda_n|\} < (2+o(1))\sqrt{np}.
\end{align}
If $p > n^{-1}\log^6 n$, then w.h.p.
\begin{align}
\label{eq:eigvecv1}
\max_j v_{1,j} < (1+o(1))\min_j v_{1,j}.
\end{align}
\end{prop}

\subsection{Path counts}
\label{subsec:Paths}

\mn
This section discusses what can be said about the
numbers of paths of various lengths joining pairs of vertices in a random graph.

\mn
\emph{Notation.} For $l \geq 1$ and (distinct) $x,y \in V$,
we use $P^l(x,y)$ for the set of $P_l$'s ($l$-edge paths) in $G$ joining $x$ and $y$,
$\tau^l(x,y) $ for $ |P^l(x,y)|$, and $\sigma^l(x,y) $ for
the maximum size of a collection of internally disjoint $P_l$'s
of $G$ joining $x$ and $y$.
(Though $l=1$ is uninteresting, it's convenient to allow this.)

In this section
we use $V(P)$
for the set of {\em internal} vertices of a path $P$ and
write $\gG^l_{x,y}$ for the graph on $P^l(x,y)$ with $P\sim Q$
iff $V(P)\cap V(Q)\neq\0$.

\medskip
Conveniently, most of what we need here has been worked out
(in far greater generality)
by Joel Spencer
in \cite{Spencer}
(see also \cite[Section 8.5]{AS}), and we begin with two special cases of
what's proved there.
\begin{thm}
\label{thm:Spencer}
For any $l\geq 2$ and $\eps>0$ there exists $K$ such that if $n^{l-1}p^l \geq K \log n$, then w.h.p.
\begin{align}
\label{eq:Spencer}
\tau^l(x,y) = (1 \pm \eps) n^{l-1}p^l ~~~~ \forall \, \{x,y\}\in \Cc{V}{2}.
\end{align}
\end{thm}

\begin{prop}
\label{prop:tausig}
For any $l \geq 1$ and $\gd>0$, if
$n^{2l-3}p^{2l-1}<n^{-\gd}$ then w.h.p.
\begin{align}
\label{eq:tausig}
\tau^i(x,y)-\sigma^i(x,y) < C ~~~~ \forall \, \{x,y\}\in\Cc{V}{2}, ~ i \in [l],
\end{align}
where $C$ depends only on $l$ and $\gd$.
\end{prop}

\nin
We note for use below that the assumption on $p$ in Proposition~\ref{prop:tausig} implies
\beq{nl2}
n^{l-2}p^{l-1} < n^{-\gz},
\enq
with $\gz = (1+\gd(l-1))/(2l-1)$ ($=\gO(1)$).
Strictly speaking, the proposition is a little stronger than what one gets from
\cite{Spencer}, where the assumption would be
$n^{l-1}p^l = O(\log n)$.
(The
$n^{2l-3}p^{2l-1}$ is more or less the expected number of
non-edge-disjoint pairs of paths joining a given $x$ and $y$.)

\medskip
Proposition~\ref{prop:tausig}, though not difficult,
is a key point in Spencer's proof of Theorem~\ref{thm:Spencer},
and from our perspective is in a sense the main point, since, as
indicated in the remark below, it easily gives the latter
when combined with Lemma~\ref{JUB}.
Since the proof of the proposition itself is not so easy to extract from Spencer's presentation
(see his ``third part" on p. 253), we next sketch an argument along lines similar to his
for the present situation.

\begin{proof}[Proof of Proposition~\ref{prop:tausig}.]
It is enough to handle $i=l$ (since the assumption on $p$ implies
a stronger assumption when we replace $l$ by $i<l$).
Noting that $\tau^l(x,y)-\gs^l(x,y)\leq |E(\gG^l_{x,y})|$,
we find that \eqref{eq:tausig} (with an appropriate $C$) holds at $x,y$ provided
\begin{itemize}
\item[(i)] the maximum number of vertices in a component of $\gG^l_{x,y}$ is $O(1)$ and
\item[(ii)] the maximum size of an induced matching in $\gG^l_{x,y}$ is $O(1)$;
\end{itemize}
so we want to say that w.h.p. these conditions hold for all $x,y$.
(Of course replacing (i) by an $O(1)$ bound on degrees would also suffice.)

\medskip
For (i) we show that, for some fixed $M$, w.h.p. there do not exist
$x,y$ and a collection, $Q_1\dots Q_M$, of $P_l$'s joining $x$ and $y$ such that, for
$i\geq 2$, $V(Q_i)$ meets, but is not contained in,
$\cup_{j<i}V(Q_j)$.
This bounds (by $(l-2)M+1$) the number of internal vertices (of $G$) in the paths belonging
to a component of $\gG^l_{x,y}$, so gives (i).

Suppose $Q_1\dots Q_M$ are $P_l$'s joining $x$ and $y$, with
$R_i=\cup_{j\leq i}Q_j$ and, for $i\geq 2$,
$|E(Q_i)\sm E(R_{i-1})|=b_i$ and
$|V(Q_i)\sm V(R_{i-1})|=a_i\in [1,l-2]$.
Then
$b_i\geq a_i+1$ and $a_i\leq l-2$ imply
$n^{a_i}p^{b_i}\leq n^{l-2}p^{l-1}$
(for $i\geq 2$)
and
\beq{nVRM}
n^{|V(R_M)|}p^{|E(R_M)|} \leq np(n^{l-2}p^{l -1})^M,
\enq
which is thus an upper bound on the probability of finding, for a given $x,y$,
$(Q_1\dots Q_M)$ as above of a given isomorphism type (defined in the obvious way).
So
the probability that there are such $Q_i$'s
for {\em some} $x,y$ (and some isomorphism type) is
$O(n^3p(n^{l-2}p^{l -1})^M )=O(n^3pn^{-\gz M})$
(see \eqref{nl2}),
so is $o(1)$ for large enough $M$.

\medskip
The argument for (ii) is similar.
Here we want to rule out, again for some fixed $M$, existence of $P_l$'s, say
$Q_1,R_1\dots Q_M,R_M$, joining some specified $x,y$, with $V(Q_i)\cap V(R_i)\neq\0$ and
the $V(Q_i)$'s and $V(R_i)$'s otherwise disjoint.
A discussion like the one above shows that for any such sequence, with
$|\cup_i(E(Q_i)\cup (E(R_i))|=b$ and $|\cup_i(V(Q_i)\cup (V(R_i))|=a$, we have
\beq{napb}
n^a p^b < (n^{2l-3}p^{2l-1})^M < n^{-M\gd},
\enq
which bounds the probability of existence by $O(n^{2-M\gd})$.
\qedhere
\end{proof}

\nin
{\em Remark.}
The lower bound in Theorem~\ref{thm:Spencer} is given by
Theorem~\ref{TJanson} (a recent development at the time).
The main issue for the upper bound is handling $p$ with
$n^{l-1}p^l \asymp\log n$, for which
Proposition~\ref{prop:tausig} allows replacing $\tau^l$ by $\gs^l$.
This is then naturally handled by Lemma~\ref{JUB},
replacing Spencer's nice, if slightly \emph{ad hoc} approach
based on maximal disjoint families.

\medskip
Theorem~\ref{thm:Spencer} and Proposition~\ref{prop:tausig}
(with bits of Section~\ref{Deviation and correlation})
easily imply the following bounds on the
$\tau^l(x,y)$'s.

\begin{cor}\label{ubtau}
W.h.p. for all (distinct) vertices $x,y$,
\begin{align}
\tau^l(x,y) \sim n^{l-1}p^l ~~
&~~\mbox{if $n^{l-1}p^l = \omega(\log n) $},
\label{eq:tausummarybig} \\
\tau^l(x,y) = O(\log n) ~~
&~~\mbox{if $n^{l-1}p^l = O(\log n) $},
\label{eq:tausummarymiddle} \\
\tau^l(x,y) = O(1) ~~
&~~\mbox{if $n^{l-1}p^l <n^{-\gO(1)}$}.
\label{eq:tausummarysmall}
\end{align}
\end{cor}

\begin{proof}
The first two items are easy consequences of Theorem~\ref{thm:Spencer}:
\eqref{eq:tausummarybig} is immediate and
\eqref{eq:tausummarymiddle} is given by the observation that, for $K$ as in
the theorem (for some specified $\eps$) and $p_0$ defined by
$n^{l-1}p_0^l = K \log n$, the theorem implies that w.h.p.
\beq{cor:taulogn'}
\tau^l(x,y) \leq (1 + \eps) n^{l-1}(\max\{p,p_0\})^l ~~~~ \forall \, \{x,y\}\in \Cc{V}{2}
\enq
(since the probability of the event in \eqref{cor:taulogn'} decreases as $p$ increases
below $p_0$).

For
\eqref{eq:tausummarysmall}, suppose $n^{l-1}p^l <n^{-\ga}$,
with $\ga>0$ fixed.  Since this implies $n^{2l-3}p^{2l-1}<n^{-\gd}$
with $\gd=\gd_\ga>0$ fixed,
Proposition~\ref{prop:tausig} says it suffices to show that for given $x,y$
and suitable fixed $D$ (depending on $\ga$),
\[ %beq{eq:tau>Dunlikely}
\Pr(\gs^{l}(x,y)>D) = o(n^{-2}).
\]  %enq
But \eqref{ET} bounds this probability by
%\[n^{-\ga D}/D! <
$\exp[-D\log(n^\alpha D/e)]$,
which is $o(n^{-2})$ for large enough $D$.
\qedhere
\end{proof}

\medskip
We will also sometimes need {\em lower} bounds on path counts, as summarized in
the next result, which again follows easily from what we already know.

\begin{cor}
\label{cor:pip}
For any $l\geq 2$ there is a $K$ such that if $n^{l-1}p^l \geq K \log n$, then w.h.p.
$\sigma^l(x,y)=\Omega(\pi)$ for all $x,y$, with $\pi = \pi(n,p)$ equal to
\begin{align}
n^{l-1}p^l& \quad \text{ if } ~n^{l-2}p^{l-1} <n^{-\gO(1)}, \label{eq:pi1}\\
n^{l-1}p^l/\log n& \quad \text{ if } ~n^{-o(1)}< ~n^{l-2}p^{l-1} =O(\log n),\label{eq:pi2}\\
np& \quad \text{ if } ~n^{l-2}p^{l-1} =\go(\log n).\label{eq:pi3}
\end{align}
\end{cor}
\nin
(Of course in view of Proposition~\ref{prop:routine}, $(1+o(1))np$ is a trivial upper bound.)

\begin{proof}
Let $K$ be as in
Theorem~\ref{thm:Spencer}, for the given $l$ and, say, $\eps =1/2$
(since we don't worry about constants).
Since the theorem says that w.h.p. $|V(\gG^l_{x,y})| >\gO(n^{l-1}p^l)$
for all $x,y$, the present assertion(s) will follow if we show
\beq{gDsmall}
\mbox{w.h.p. $~\gD(\gG^l_{x,y}) =O(n^{l-1}p^l/\pi) ~\forall x,y$,}
\enq
where we use the
the trivial $\ga\geq |V|/\gD$
(recall $\gD$ and $\ga$ are maximum degree and independence number and note
$\gs^l(x,y) = \ga(\gG^l_{x,y})$).
Now the degree in $\gG^l_{x,y}$ of a given vertex $Q$
(that is, a $P_l$ joining $x$ and $y$) is at most
\beq{deltabd}
\mbox{$\sum_v\sum_i\tau^i(x,v)\tau^{l-i}(v,y)\leq
(l-1)^2\max\{\tau^i(x,v)\tau^{l-i}(v,y)\},$}
\enq
where the sums are over $v\in V(Q)$
and $i\in [l-1]$, and
the max is over $i\in [l-1]$ and $v\in V\sm \{x,y\}$
(the initial $(l-1)^2$ is of course irrelevant).
On the other hand,
Corollary~\ref{ubtau}
(with $i$ in place of $l$)
says that w.h.p. we have, for all $u,v$:

\begin{quote}
$\tau^i(u,v) < O(1)~$
if either $i\leq l-2$ and $p$ is as in \eqref{eq:pi1} or \eqref{eq:pi2},
or $i=l-1$ and
$p$ is as in \eqref{eq:pi1},
\end{quote}

\nin
and
$\tau^i(u,v) < O(\max\{n^{i-1}p^i, \log n\})$ in general;
and combining these bounds with \eqref{deltabd} easily yields \eqref{gDsmall}.
\end{proof}

Finally, in connection with the setup introduced at \eqref{eq:f}, we will need
the following simple observation:
\beq{xyF}
xy\in F ~~ \Longrightarrow ~~ |F|\geq \gs^{\kappa-1}(x,y)+1.
\enq
\begin{proof}
Since $F $ lies in $\cee^\perp_{\kappa}(G)$, it must contain a second edge of each $\kappa$-gon
of $G$ containing $xy$, and there is a set of $\gs^{\kappa-1}(x,y)$ such $\kappa$-gons that
share no edges except $xy$.
\end{proof}

\subsection{Stability}\label{Stability}

The following statement is an instance of a major result of
Conlon and Gowers \cite{Conlon-Gowers}. As mentioned in Section \ref{SEC:Overview},
this is the main (essentially only) ingredient in the proof of Lemma~\ref{LemCG}
given in Section \ref{SEC:PLCG}.

\begin{thm}
\label{thm:ConlonGowers}
For each odd $\kappa \geq 3$ and $\eps>0$ there is a $C$ such that if
$p > Cn^{-(\kappa-2)/(\kappa-1)}$, then w.h.p. every $C_\kappa$-free subgraph
of $G=G_{n,p}$ of size at least
$|G|/2$ can be made bipartite by deleting at most $\eps n^2p$ edges.
\end{thm}
\nin
This (or the more general result of \cite{Conlon-Gowers})
is a ``sparse random'' analogue of the Erd\H{o}s-Simonovits
``Stability Theorem'' \cite{Erdos,Simonovits} that was
conjectured by Kohayakawa {\em et al.} in the seminal
\cite{KLR}.

\medskip
As mentioned in Section~\ref{SEC2}, Lemma~\ref{LemCG}
can be considerably extended; in fact we can prove something similar with
$C_\kappa$ replaced by a general $H$, though not always with the
%(conjecturally correct)
lower bound on $p$ that would correspond to a positive answer to Question~\ref{Gen'l?}.
See
Section~\ref{SEC:PLCG} for a precise statement.
%and Section~\ref{CGA} for a (sketchy) proof.}

\subsection{Coupling}
\label{subsec:Coupling}

A central role in the proofs of
Lemmas~\ref{lem:p=O(p)'} and \ref{lem:couplingdown}
is played by the usual coupling of $G:=G_{n,p}$ and $G_0:=G_{n,q}$,
where $p$ will always be the value we're really interested in and $q<p$ will depend
on what we're trying to do.
A standard description:

Let $\gl_e$, $e\in E(K_n)$, be chosen uniformly and independently from $[0,1]$ and set
\[
G = \{e:\gl_e<p\}, ~~~ G_0 = \{e:\gl_e<q\}.
\]
In particular $G_0\sub G$.
Probabilities in the proofs of Lemmas~\ref{lem:p=O(p)'} and \ref{lem:couplingdown} refer
to the joint distribution of $G$ and $G_0$.

We will get most of our leverage from two alternate ways of viewing the choice of
the pair $(G,G_0)$:

\begin{itemize}
\item[(A)]  Choose $G$ first; thus we choose $G$ ($=G_{n,p}$) in the usual way and let $G_0$
be the (``($q/p$)-random") subset of $G$ gotten by retaining edges of $G$ with probability
$q/p$, these choices made independently (a.k.a.\ {\em percolation on} $G$).

\item[(B)] Choose $G_0$ first; that is, we choose $G_0$ ($=G_{n,q}$) in the usual way,
define $p'$ by $(1-q)(1-p')=1-p$, and let $G$ be the random superset of $G_0$ gotten by
adding each edge of $\overline{G}_0$ to $G_0$ with probability $p'$, these choices again
made independently.
\end{itemize}

\nin
We will often refer to these as ``coupling down" and ``coupling up" (respectively).

\medskip
The proof of
Lemma~\ref{lem:couplingdown} is based naturally (or inevitably) on the viewpoint in (A);
namely, we show that (with $p,q$ as in the lemma) if $G=G_{n,p}$ is ``bad"
(meaning $G\not\in \T$) then
the coupled $G_0=G_{n,q}$ is likely to be bad as well.
For the proof of
Lemma~\ref{lem:p=O(p)'}, viewpoint (B) is the primary mover, though the role of
(A) is also crucial.

With reference to the setup introduced at \eqref{eq:f},
when working with $G=G_{n,p}$ and $G_0=G_{n,q}$ as above, we set
$F_0=G_0 \cap F$
(a $(q/p)$-random subset of $F$; note this has nothing to do with $F(G_0)$,
which will play no role here).
Then automatically
\beq{F0eperp}
F_0 \in \cee^\perp_{\kappa}(G_0),
\enq
since $F_0\cap C=F\cap C$ for any $\kappa$-gon $C$ of $G_0$.

We will want to say that certain features of $(G,F)$ are reflected in
$(G_0,F_0)$.
A simple but crucial point here is that there is no summing (of probabilities) over possible
$F$'s, since there is just one $F$ for each $G$.
The following proposition will be sufficient for our purposes.

\begin{prop}
\label{prop:Frightsize}
With the above setup,
for any $p$, $q$ and $g=g(n)=\go(1)$, w.h.p.
\[
|F_0|\sim |F|q/p ~~\mbox{if} ~~ |F| > gp/q
\]
and
\[
d_{F_0}(v) \left\{\begin{array}{ll}
\sim d_F(v)q/p&\mbox{$\forall v$ with $d_F(v) > (g \log n )p/q$,}\\
<3g\log n &\mbox{$\forall v$ with $d_F(v) \leq (g \log n )p/q$.}
\end{array}\right.
\]
\end{prop}

\nin
(This is true for any rule that specifies a particular subgraph (in place of $F$)
for each graph;
but we will only use it with $F$ ($=F(G)$), so just give the statement for this case.)

\begin{proof}
These are straightforward applications of Theorems~\ref{thm:Chernoff} and
\ref{thm:Chernoff'}, so we will be brief.
For the first assertion we want to say that for any fixed $\eps>0$,
\[
\Pr \big((|F|> gp/q)\wedge (|F_0| \neq (1\pm \eps)|F|q/p)\big) \ra 0.
\]
But the probability here is less than
\[
\Pr \big(|F_0| \neq (1\pm \eps)|F|q/p \; \mid \; |F|> gp/q \big),
\]
which by Theorem~\ref{thm:Chernoff} is less than
$\exp[-\gO(\eps^2g)]$.

The second assertion (pair of assertions) is similar, following from
\begin{align*}
\sum_v\Pr\big(d_{F_0}(v)\neq (1\pm \eps)d_F(v) \;\big|\; d_F(v) > (g\log n) p/q \big)
&< n \exp[-\gO(\eps^2g\log n)] \\
&= o(1)
\end{align*}
for any fixed $\eps>0$, and (now switching to Theorem~\ref{thm:Chernoff'})
\begin{align*}
\sum_v\Pr \big(d_{F_0}(v)>3g\log n \;\big|\; d_F(v) \leq (g \log n )p/q \big)
&< n \exp[-(3g\log n)\log (3/e) ]\\
&= o(1). \qedhere
\end{align*}
\end{proof}

\section{Two simple points}
%Proof of Lemma \ref{lem:Qsharpthresh}}
\label{SEC:LemmaThresh}

Here we dispose of Lemma~\ref{lem:Qsharpthresh} and the derivation of
Theorem~\ref{thm:C5threshold} from Theorem~\ref{thm:C5precise}.
(Recall we are using $G$ for $G_{n,p}$ and $V$ for $V(G)$.)

\begin{proof}[Proof of Lemma~\ref{lem:Qsharpthresh}.]
We begin with the 1-statement,
a typical application of Theorem~\ref{TJanson}.
We assume $p>(1+\eps)p^*$ and $p=O(p^*)$ (as we may, since
for larger $p$, the 1-statement is contained in Theorem~\ref{thm:Spencer}).
Given $x,y\in V$, let the $A_i$'s (in the paragraph preceding Theorem~\ref{TJanson})
be the edge sets of the $(\kappa-1)$-paths joining $x$ and $y$ in $K_n$;
so $X=\tau^{\kappa-1}(x,y)$,
$\mu \sim n^{\kappa-2}p^{\kappa-1}$ and
$\overline{\Delta} = \mu + O(\mu n^{\kappa-3}p^{\kappa-2}) \sim \mu$.
Thus (note $\varphi(-1)=1$) Theorem~\ref{TJanson} gives
\begin{align}
\label{eq:LemmaThreshPaths}
\Pr(\tau^{\kappa-1}(x,y)=0) %= \Pr(\tau^{\kappa-1}(x,y) \leq \mu-\mu)
\leq \exp[-(1-o(1))\mu].
\end{align}
So the probability that $\Q$ ($=\Q_\kappa$) fails---that is, that there is some
$xy$ in $G$ with $\tau^{\kappa-1}(x,y)=0$---is less than
\[
\Cc{n}{2}pe^{-(1-o(1))\mu} <\exp[\log(n^2p) - (1-o(1))\mu] =o(1)
\]
(since
$\mu > (1-o(1))(1+\eps)^{\kappa-1}(\kappa/(\kappa-1))\log n \sim (1+\eps)^{\kappa-1}\log(n^2p)$).

\medskip
For the 0-statement we use the second moment method (see e.g. Chapter 4 of \cite{AS})
and, again, Theorem~\ref{TJanson}.
Let $Z_{xy}$ be the indicator of the event $\{xy \in G\} \wedge \{\tau^{\kappa-1}(x,y)=0\}$
($x,y\in V$)
and $Z=\sum Z_{xy}$. Theorem \ref{thm:Harris} gives
$\Pr(\tau^{\kappa-1}(x,y)=0) > (1-p^{\kappa-1})^{n^{\kappa-2}} > \exp[-\mu-o(1)]$
($\mu$ as above),
whence
\begin{align}
\label{eq:LemmaThresh0}
\mathbb{E}[Z_{xy}] > p\exp[-\mu-o(1)].
\end{align}
In particular $\mathbb{E}[Z] = \omega(1)$ (using $p < (1-\eps)p^*$ and ignoring the rather trivial
case $p=O(n^{-2})$),
so for
$\mathbb{E}Z^2\sim \mathbb{E}[Z]^2$ (which gives the 0-statement {\em via} Chebyshev's Inequality),
it's enough to show
\[
\mathbb{E}[Z_{xy}Z_{uv}] < (1+o(1))\mathbb{E}[Z_{xy}]^2
\]
for distinct $\{x,y\}$,$\{u,v\}\in \binom{V}{2}$,
which in view of \eqref{eq:LemmaThresh0} follows from
\begin{align*}
\mathbb{E}[Z_{xy}Z_{uv}] &\leq p^2\Pr(\tau^{\kappa-1}(x,y)=\tau^{\kappa-1}(u,v)=0) \\
&\leq p^2\exp[-(1-O(n^{\kappa-3}p^{\kappa-2}))2\mu]
=p^2\exp[-2\mu+o(1)].
\end{align*}
Here the first inequality is given by Theorem \ref{thm:Harris}
(since the events $\{xy, uv \in G\}$ and $\{\tau^{\kappa-1}(x,y)=\tau^{\kappa-1}(u,v)=0\}$
are increasing and decreasing respectively),
and the second by Theorem~\ref{TJanson}, where the $A_i$'s are the $(\kappa-1)$-edge paths
joining either $x$ and $y$ or $u$ and $v$, for which $\E X\sim 2\mu$ (recall $X$ is the number
of $A_i$'s that occur) and
it's easy to see that
$\ov{\gD} -\mu=O(n^{2\kappa -5}p^{2\kappa-3})= O(n^{\kappa-3}p^{\kappa-2})\mu$ ($=o(\mu)$).
\end{proof}

\mn
\begin{proof}[Proof that Theorem~\ref{thm:C5precise} implies Theorem~\ref{thm:C5threshold}.]
This is again routine and we aim to be brief.
Lemma~\ref{lem:Qsharpthresh} gives the 1-statement
(which is the interesting part).  For the 0-statement,
it is enough to say that for $p$ in the stated range, $G=G_{n,p}$ w.h.p. contains an edge lying
in a cycle but not in a $C_\kappa$.
This is again given by Lemma~\ref{lem:Qsharpthresh} if $p$ is large enough that \emph{all} edges
are in cycles (w.h.p), which is true if $p> (1+\gO(1))\log n/n$ (again, see \cite[p.\ 105]{JLR}).
For smaller $p$, w.h.p. $G$ contains cycles of length $\go(1)$ if $p> (1-o(1))/n$
and of length (say) $\gO(n^{.3})$ if $p> (1+\gO(1))/n$
(see e.g. \cite[Thm. 5.18(i)]{JLR}).  On the other hand, since the expected number of
$C_\kappa$'s in $G$ is less than $(np)^\kappa$, the number of edges in $C_\kappa$'s is
w.h.p. less than $\go \cdot(np)^\kappa$ for any $\go=\go(1)$;
so in the range under discussion, the $C_\kappa$'s w.h.p.  don't cover
the edges of even one longest
cycle in $G$.
\end{proof}

\section{Proof of Lemma~\ref{LemCG}}
\label{SEC:PLCG}

Here we give the easy proof of Lemma~\ref{LemCG} and then state
the extension to general $H$ mentioned in the remark
following the lemma.

For the lemma it's enough to show that the conclusions of
Proposition~\ref{prop:routine}, Theorem~\ref{thm:ConlonGowers} and
Proposition~\ref{density}(c), the latter two with $\eps =c/3$,
imply $|F|< cnp^2$
(deterministically).

Let $F'$ be a largest element of $F + \mathcal{C}^\perp(G)$.
Then $|F'| \geq |G|/2$
(by Proposition~\ref{prop:addcut}), so, since $F'$ is $C_\kappa$-free,
the conclusion of Theorem~\ref{thm:ConlonGowers} gives
an $A \subseteq V$ with
\beq{F'setminus}
|F' \setminus \nabla_G(A)| < \eps n^2p.
\enq
To finish we just check that (under our assumptions),
\eqref{F'setminus} implies
\[
\mbox{($|F|\leq $) $~~|F'\triangle \nabla_G(A)| < 3\eps n^2p~$:}
\]
the conclusion of
Proposition~\ref{density}(c)
gives
$|\nabla_G(A)| < (1+\eps)n^2p/4$, whence
\[
|\nabla_G(A) \setminus F'| \leq (1+\eps)n^2p/4 - (|G|/2 - \eps n^2p) < 2\eps n^2p
\]
(where we again used
Proposition~\ref{prop:routine} to say $|G|\sim n^2p/2$).\qed

\bn
{\em Generalization.}
(We continue to use $G$ for $G_{n,p}$.)
For this discussion we restrict to $H$ with $e_H\geq 2$ (so $v_H\geq 3$).
For such an $H$, set
\beq{m2}
m_2(H) = \max\left\{\frac{e_K-1}{v_K-2}: K\sub H,v_K\geq 3\right\}.
\enq
This parameter plays a central role in various contexts,
in particular in results more or less related to (the general version of) Theorem~\ref{thm:ConlonGowers};
see e.g. \cite{Rodl-Schacht} for an overview.

\begin{thm}\label{TCGA}
For any fixed $H$ the following is true.
For any $\eps>0$ there is a C such that if $p>C n^{-1/m_2(H)}$ then w.h.p.:
for each $F\in \cee_H^\perp(G)$ there is an $X\in \W_H^\perp(G)$ with
$|F\Delta X|< \eps n^2p$;
in particular, if $\cee_H(G)\neq \W_H(G)$, then
\[
\min\{|F|:F\in \cee^\perp_H(G)\sm \W^\perp_H(G)\}< \eps n^2p.
\]
\end{thm}
\nin
Since we aren't using this (and since the present work is already too long),
we refer to \cite[Sec.\ 4.8]{Baron} for the proof, 
here just mentioning that the main ingredients are
the ``container" machinery of \cite{BMS,ST}
and the following analogue of
the Erd\H{o}s-Simonovits
``Stability Theorem" \cite{Erdos,Simonovits}. 
% mentioned following Theorem~\ref{thm:ConlonGowers}.
(The role of this lemma in the proof of Theorem~\ref{TCGA}
is similar to that of Erd\H{o}s-Simonovits
in the proofs of Theorem~\ref{thm:ConlonGowers} in
\cite{BMS,ST}.)

For any $H$ and $F\sub E(K_n)$, let $\tau_H(F)$
be the number of copies of $H$ in $K_n$ (say unlabelled)
having odd intersection with $F$.

\begin{lemma}\label{ESish}
For any fixed graph H and $\eps>0$, there is a $\gd>0$ such that if
$F\sub E(K_n)$
satisfies $\tau_H(F)<\gd n^{v_H}$, then there is an $X\in \W_H^\perp(K_n)$
with $|F\Delta X|< \eps n^2$.
\end{lemma}

\mn
{\em Remarks.}
Notice that Theorem~\ref{TCGA}
{\em contains} an extension of Lemma~\ref{LemCG},
whereas in the preceding discussion we did need a few lines to get
from Theorem~\ref{thm:ConlonGowers} to the lemma.
But the two theorems live in somewhat different worlds, since
Theorem~\ref{thm:ConlonGowers} assumes only that $F$ is $C_\kappa$-{\em free},
which is much weaker than requiring that it have even intersection with every $C_\kappa$.

As mentioned in Section~\ref{Stability}, the value $n^{-1/m_2(H)}$ is not
necessarily what's needed for Question~\ref{Gen'l?}.
For instance, if $H$ is two triangles joined by a $P_l$, then $m_2(H)=2$
(take $K$ to be one of the triangles), but the range where the question
is most interesting
(the point at which $\Q_H$ becomes likely)
is $p\asymp n^{-2/3}\log ^{1/3}n$, corresponding to all {\em vertices} being in triangles.
%turns out to have
%$p =\tilde{\Theta}(n^{-(l+3)/(l+5)})$.
On the other hand, in natural cases---e.g.
the (``balanced") $H$'s for which $K=H$ achieves the max in
\eqref{m2}---Theorem~\ref{TCGA} does give what should be the correct
extension of Lemma~\ref{LemCG}.
(It would be interesting to see if one could push the theorem to give
the correct extension in general; with our current approach this
would mainly require a fairly significant extension of
what we are getting from ``containers," and we haven't yet thought about plausibility.)

\bn

\section{Proof of Lemma \ref{lem:couplingdown}}
\label{SEC:LemmaCoupling}

\mn
By Corollary~\ref{cor:pip} with $l=\kappa-1$, there is a $K>1$ such that if $p > Kp^*$, then w.h.p.
\begin{align}
\label{eq:sigmaLB}
\text{every $\{x,y\}\in\Cc{V}{2}$ satisfies $\sigma^{\kappa-1}(x,y) = \Omega(\pi)$}
\end{align}
(where $\pi=\pi(n,p)$ is as in the corollary).
%and, to emphasize, $\sigma^{\kappa-1}(x,y)$ refers to paths in $G=G_{n,p}$.
We work in the coupling framework of
Section~\ref{subsec:Coupling}, taking $q = Kp^*$ and $G_0=G_{n,q}$.

% with $p>q=2p^*$ as in Lemma~\ref{lem:couplingdown}.

For Lemma \ref{lem:couplingdown} it is of course enough to show
\beq{fGfG0}
\Pr(\{G \notin \mathcal{T}\} \; \wedge \; \{G_0 \in \mathcal{T}\}) \ra 0.
\enq
%(since $\{G \notin \mathcal{T}\} \subseteq \{G_0 \notin \mathcal{T}\}
%\cup \{(G \notin \mathcal{T}) \wedge (G_0 \in \mathcal{T})\}$).
Note that $G_0\in \T$
%the event in \eqref{fGfG0}
implies $F_0\in \cee^\perp(G_0)$,
since we always have $F_0\in \cee^\perp_{\kappa}(G_0)$
(see \eqref{F0eperp}); thus \eqref{fGfG0} will follow from
\beq{fGF0}
\Pr(\{F \neq \0 \} \; \wedge \; \{F_0 \in \cee^\perp(G_0)\})\ra 0.
\enq
So it will be enough to show that
\beq{F0cee}
F_0 \notin \cee^\perp(G_0)
\enq
follows (deterministically) from
\beq{F0}
F \neq \0
\enq
combined with various statements that we already know hold w.h.p.
This is not hard, but is more circuitous than one might wish.
Roughly we show that, barring occurrence of some low probability event,
(i) presence of even one edge in $F$ forces $F$ to be large enough
(not very large) that $F_0 \neq \0 $, and
(ii)
$F_0$ is not substantial enough
to meet all $xy$-paths in $G_0-xy$ for an $xy\in F_0$, so any such
$xy$ is contained in a cycle witnessing \eqref{F0cee}.

\mn
{\em A convention.}
To slightly streamline the presentation we agree
that in this argument, appeals to a probabilistic statement $X$---e.g.
``$X$ implies" or ``by $X$"---actually refer to
{\em the conclusion of} $X$,
which conclusion will always be something that $X$ says holds w.h.p.
See the references to \eqref{eq:sigmaLB}, Lemma~\ref{LemCG}
and Proposition~\ref{prop:Frightsize} in the next paragraph for first instances
of this.

\medskip%Fix $\vt \in (0, 5^{-\kappa})$.
If \eqref{F0} holds, then
\eqref{eq:sigmaLB} and \eqref{xyF}
(for the lower bound) together with
Lemma~\ref{LemCG} (for the upper)
imply that
\beq{couple1}
\Omega(\pi) <|F| < n^2p/10.
\enq
Since $\pi q/p \gg 1$, the lower bound in \eqref{couple1} and the first part
of Proposition~\ref{prop:Frightsize} give $|F_0|\sim |F|q/p$, so
\beq{couple2}
0\neq |F_0| < (1+o(1)) n^2q/10.
\enq

\nin
In addition, Proposition~\ref{prop:routine},
\eqref{dFdG} and the second part of Proposition~\ref{prop:Frightsize} give
\[
d_{F_0}(v) <(1+o(1))nq/2 ~~\forall \, v \in V.
\]
Thus, setting $H_0=G_0\sm F_0$ and recalling the approximate ($nq)$-regularity of $G_0$
given by Proposition~\ref{prop:routine}, we have
\beq{couple3}
d_{H_0}(v) > (1-o(1))nq/2 ~~\forall \, v \in V.
\enq

Now choose an $xy\in F_0$ (recall \eqref{couple2} says $F_0\neq \0 $)
and let $X,Y$ be the $H_0$-components of $x$ and $y$.  By \eqref{couple3} and
Proposition~\ref{cpts} (applied to $G_0$), we have $|X|,|Y| > n/3$, which implies $X=Y$:
otherwise $X$ and $Y$ are disjoint and we have the contradiction
\[
(1-o(1))n^2q/9 < |\nabla_{G_0}(X,Y)| \leq |F_0| < (1+o(1)) n^2q/10,
\]
where the first inequality is given by Proposition~\ref{density}(a) (applied to $G_0$),
the second holds because $\nabla_{G_0}(X,Y)\sub F_0$,
%(since $X$ and $Y$ are distinct components of $H_0$),
and the third is given by \eqref{couple2}.

But this (i.e.\ $X=Y$) gives an $xy$-path in $H_0$, and adding $xy$ to this path
produces a cycle meeting $F_0$ only in $xy$; so we have \eqref{F0cee}.

\section{Proof of Lemma~\ref{lem:p=O(p)'}}
\label{SEC:LemmaP=O(p*)}

\mn
Here we first introduce the main assertions, Lemmas~\ref{lem:Claim2}
and \ref{lem:InterruptedPathsBound}, underlying
Lemma~\ref{lem:p=O(p)'}, and prove the latter assuming them.
The supporting lemmas are then proved in Sections \ref{subsec:Lemma6.1} and \ref{subsec:Lemma6.2}.

Note that for the proof of Lemma~\ref{lem:p=O(p)'},
Lemma~\ref{lem:Qsharpthresh} allows us to restrict attention to the range
\begin{align}
\label{eq:pgeq1-eps}
(1-\eps)p^*< p< Kp^*
\end{align}
(for any fixed
$\eps>0$), and that Lemma~\ref{LemCG} says
it's enough to show that for a given $\lambda=\lambda(n) \ra 0$,
\begin{align}
\label{eq:ETS1}
\Pr(\{G_{n,p} \in \Q\} \; \wedge \; \{0 < |F| < \lambda n^2p\}) \ra 0.
\end{align}

We again work with the coupling of Section~\ref{subsec:Coupling},
now taking $q=\vartheta p$ with a {\em fixed} $\vartheta \in (0,1)$
small enough to support the discussion below
(the rather mild constraints on $\vt$ are at
\eqref{eq:GminusG0capR} and \eqref{vt.constraint2}).
Define the random variables $\alpha$ and $\ga_0$ by
\begin{align}
\label{eq:defalpha}
\mbox{$|F| = \alpha n^2p/2~$ and $~|F_0| = \alpha_0 n^2q/2$.}
\end{align}

\mn
\emph{Definitions.} Henceforth a \emph{path} (with length unspecified) is a $P_{\kappa-1}$
(and an $xy$-path is a path whose endpoints are $x$ and $y$).
Our paths will always lie in $G$ and often in $G_0$.
We now write $\sigma(x,y)$ for
$\sigma^{\kappa-1}(x,y)$ (recall from
Section \ref{subsec:Paths} that this is the maximum size of a set of internally
disjoint $xy$-paths in $G$), and $\sigma_0(x,y)$ for the analogous quantity in $G_0$.
For $S \subseteq G$, a path $P$ is $S$-\emph{central}
if it contains an odd number of edges of $S$, at
least one of which is internal.
Let $\sigma(x,y;S)$ be the maximum size of a collection of internally disjoint
$S$-central $xy$-paths, and $\sigma_0(x,y;S)$ the corresponding quantity in $G_0$.
%The default for $S$ is $F$, and in this case we abbreviate $\sigma(x,y,F)$ by $\sigma'(x,y)$ and $\sigma_0(x,y,F)$ by $\sigma_0'(x,y)$.
An $(S,t)$-\emph{rope} is a $P_t$ whose terminal edges lie in $S$.
Set
\beq{R}
R(S) = \{ \{x,y\} \in \Cc{V}{2} : \sigma_0(x,y;S) > .25n^{\kappa-2}q^{\kappa-1}\}
\enq
and define events
\[     %\beq{R}
\R= \{|F \cap R(F_0)| \geq .12\alpha n^2p\}
\]    %\enq
and
\[
\pee =\{0 < |F| < \lambda n^2p\}
\]
(the second conjunct in \eqref{eq:ETS1}).

\begin{lemma}
\label{lem:Claim2}
There is a fixed $\eps>0$ such that for $p$ as in \eqref{eq:pgeq1-eps}, w.h.p.
\begin{align}
\label{eq:Claim2statement'}
G \in \mathcal{Q}\wedge\pee \;\; \Rightarrow \;\;
G\in \R.
%|F \cap R(F_0)| \geq .12\alpha n^2p.
\end{align}
\end{lemma}
\nin
(In other words,
$\Pr(G \in \mathcal{Q}\wedge\pee\wedge \ov{\R}) \ra 0.$
Of course $\R$ holds trivially if $F=\0$, so it's
only the upper bound in $\pee$ that's of interest here.)

\mn
{\em Remarks.}
For $\{x,y\} \in \binom{V}{2}$, $\sigma_0(x,y)$
should be around $n^{\kappa-2}q^{\kappa-1}$.
Lemma~\ref{lem:Claim2} says that, provided $G \in \Q\wedge \pee$,
it's likely that for a decent fraction of the edges $xy$ of $F$, even
$\sigma_0(x,y,F_0)$ is of this order of magnitude---which
is
{\em un}natural if $F_0$ is small relative to $G_0$
(since then paths should typically avoid $F_0$).
Viewed from Lemma~\ref{lem:Claim2} the parity requirement in the definition of
``central" may look superfluous, since a
path of $G_0$ joining ends of an edge of $F$ necessarily has odd intersection with $F_0$;
but this extra condition
will later play a brief but important role in justifying \eqref{RSimplies}.

\medskip
For the next lemma we temporarily expand the range of $q$ and $G_0$, assuming only
what's needed for the proof (though we will use the lemma only with $q$ and $G_0$ as above).
\begin{lemma}
\label{lem:InterruptedPathsBound}
For fixed $t \geq 3$,
$\q =\q (n) > n^{-1} \log^6 n$ and $G_0=G_{n,\q }$, w.h.p.:
for $S \sub G_0$, say with $|S| =\gb n^2\q /2$,
the number of $(S,t)$-ropes in $G_0$ is
\begin{align}
\label{eq:InterruptedPathsBound}
O(\max\{\gb ^2n^{t+1}\q ^t, \; \gb n^{t/2+2} \q ^{t/2+1}\}).
\end{align}
\end{lemma}
\nin
%(Though $\q$ will later be $q$, we use the strange font here
%since $q$ has a more particular meaning.)

\mn
{\em Remarks.}
Note this is of interest only when $\gb\ll 1$,
since Proposition~\ref{prop:routine}
bounds (w.h.p.) the number in question by $(1+o(1))n^{t+1}q^t$;
see Section~\ref{subsec:Lemma6.2} for a little more on the bounds in
\eqref{eq:InterruptedPathsBound}.
The bound is also correct, but more trivial, when $t=2$.
The lemma doesn't actually require $S\sub G_0$: the proof shows that,
for any $S\sub E(K_n)$ (of the stated size) with $\gD_S =O(nq)$
(where $\gD$ is maximum degree), we have the same bound
for the number of $P_t$'s with terminal edges in $S$ and internal edges in $G_0$.

\mn
{\em Preview.}
The proof of Lemma~\ref{lem:p=O(p)'}, which we are about to give, is based mainly on
``coupling up": using information about $(G_0,F_0)$ to constrain what
happens when we choose $G\sm G_0$.
(To this extent our strategy is similar to that of \cite{DHK}, but the resemblance ends there.)
On the other hand, the proof of the crucial
Lemma~\ref{lem:Claim2} in Section~\ref{subsec:Lemma6.1}
is based on ``coupling down": most of the work there is devoted to
the proof of a similar statement (Lemma~\ref{prop:Claim1}) involving only $G$
(not $G_0$), from which the desired hybrid statement follows easily {\em via} coupling.
In sum, we couple down to show that $\R$ is likely
(precisely, the conjunction of its failure with $\Q\wedge\pee$ is unlikely),
and couple up to show it is {\em un}likely.
A little more on the latter:

We would like to say that if $G_0$ is sufficiently nice---as it will be w.h.p.---then
$\pee\wedge\R$
is unlikely; this gives \eqref{eq:ETS1} {\em via} Lemma~\ref{lem:Claim2}.
The main point we need to add to
Lemmas~\ref{lem:Claim2}
and \ref{lem:InterruptedPathsBound} is a deterministic one:  if $G_0$
enjoys relevant genericity properties, together with the
conclusion of
Lemma~\ref{lem:InterruptedPathsBound},
then, for each $S\sub G_0$,
$R(S)$ is fairly small (depending on $|S|$; see \eqref{Timplies}).
Combined with $F\neq\0$ (from $\pee$),
this will allow us to say that the lower bound on
$|G \cap R(F_0)|$ ($= |F \cap R(F_0)|$) in
$\R$ is
larger by a crucial factor $\ga^{-\gO(1)}$ than $|R(F_0)|p$---its
natural value when we ``couple up"---which {\em ought} to make $\R$ unlikely.
But of course $F_0$ depends on $G$; so, given
$G_0$, we are forced to sum
the probability of this supposedly unlikely event over possible values $S$ of
$F_0$, which turns out to mean that the whole argument would collapse if we were to
replace the above $\ga^{-\gO(1)}$ by $\ga^{-o(1)}$.
(Here we again use $\pee$, in this case to say $\ga$ is small.)

A word on presentation.
We prove the desired
\beq{QP}
\Pr(\Q\wedge \pee) = o(1)
\enq
(= \eqref{eq:ETS1})
by producing a list of unlikely events and showing
that at least one of these must hold if $\Q\wedge\pee$ does.
A
more intuitive formulation might, for example, begin:
``By Lemma~\ref{lem:Claim2} (since we assume $\Q\wedge \pee$), {\em we may assume}
$\R$."
But note this would really mean, not that we
{\em condition} on $\R$---not something we can hope to understand---but
that we need only bound probabilities $\Pr(\sss\wedge \R)$
for $\sss$'s of interest, and for a formal discussion this
seems most clearly handled by something like the present approach.

\mn

For the derivation of Lemma~\ref{lem:p=O(p)'}
we need two more events (supplementing $\pee,\Q,\R$ above).
The first of these is simply
\[
\sss = \{ \ga_0\sim \ga\}
\]
(i.e.
for any $\eta>0$, $\ga_0=(1\pm \eta)\ga$ for large enough $n$;
recall $\ga,\ga_0$ were defined in \eqref{eq:defalpha}).
The second, which we call $\T$, is the conjunction of a few properties of $G_0$
that we already know hold w.h.p., namely:
$|G_0|\sim n^2q/2$ (see Proposition~\ref{prop:routine});
\eqref{eq:tausummarybig} and
\eqref{eq:tausummarymiddle} for $l \in [\kappa -1,2\kappa -6]$
(meaning, in view of \eqref{eq:pgeq1-eps}, \eqref{eq:tausummarymiddle}
if $l=\kappa-1$ and
\eqref{eq:tausummarybig} otherwise);
and the conclusion of Lemma~\ref{lem:InterruptedPathsBound}
for $t\leq \kappa-1$ (actually we only need this for even $t$).
We first outline and then fill in details.

We will show
\beq{RS}
\Pr(\R\wedge \{F\neq\0\} \wedge \ov{\sss})=o(1).
\enq
(This is easy and a secondary use of $\R$.
Note $\{F\neq\0\}$ is implied by $\pee$.)

We will also show that (\emph{deterministically})
\beq{RSimplies}
\R\wedge\{F\neq\0\} \wedge \sss ~\Longrightarrow ~|(G \setminus G_0) \cap R(F_0)| >  .1\alpha n^2p
\enq
provided $\vt$ is sufficiently small
(this is again easy), and, as mentioned in the preview,
\beq{Timplies}
\T ~\Longrightarrow ~ |R(S)| = O(\ga_S^{1+\gd}n^2)
\enq
for some fixed $\gd>0$ and all $S\sub G_0$,
where we set $\ga_S= 2|S|/(n^2q)$.
%(Note the last two implications are deterministic.)
Thus the conjunction of $\pee,\R,\sss$ and $\T$ implies
(again, deterministically), the event---call it $\U$---that
$|G_0|< n^2q$ (say) and there is an $S\sub G_0$
(namely the one that will become $F_0$) satisfying (say):
\beq{U}
\mbox{$\ga_S<2.1\gl$, $|R(S)|=O(\ga_S^{1+\gd}n^2)$,
and $|(G \setminus G_0) \cap R(S)| >  .09\ga_S n^2p$}.
\enq

Thus, finally, for \eqref{eq:ETS1} it is enough to show (by a routine calculation)
\beq{Utoshow}
\Pr(\U) =o(1).
\enq
(Because:  since $\ov{\U}$ implies $\ov{\pee}\vee\ov{\R}\vee\ov{ \sss}\vee\ov{\T}$,
\eqref{Utoshow} implies
\[
\Pr(\Q\wedge (\ov{\pee}\vee\ov{\R}\vee\ov{ \sss}\vee\ov{\T}))=\Pr(\Q)-o(1);
\]
but the l.h.s. here is at most
\[
\Pr(\Q\wedge \ov{\pee}) + \Pr(\Q\wedge\pee\wedge\ov{\R}) +
\Pr(\pee\wedge \R\wedge \ov{ \sss})+\Pr(\ov{\T})=\Pr(\Q\wedge\ov{\pee})+o(1)
\]
(the second and third terms on the l.h.s. being bounded by Lemma~\ref{lem:Claim2} and \eqref{RS}
respectively);
so we have $\Pr(\Q\wedge \pee) =\Pr(\Q)-\Pr(\Q\wedge \ov{\pee}) =o(1)$.)

\begin{proof}[Proof of \eqref{RS}]

\mn
If $F\neq \0$ (i.e. $\alpha> 0$) and $\R$ holds, then
$F\cap R(F_0)\neq \0$, while by \eqref{xyF}, for any $xy\in F\cap R(F_0)$,
\[
|F|> \gs(x,y)\geq \gs_0(x,y)
> .25n^{\kappa-2}q^{\kappa-1} = \Omega(\log n).
\]
But then (since $\log n \gg p/q$)
Proposition \ref{prop:Frightsize} says that w.h.p.
$|F_0| \sim \vartheta |F|$, which is the same as $\sss$.\qedhere

\end{proof}

\begin{proof}[Proof of \eqref{RSimplies}]

Note it is always true that
$     %\beq{eq:G0capRlsmall}
G_0 \cap R(F_0) \subseteq F_0,
$    %\enq
since
the endpoints of an $xy \in (G_0 \cap R(F_0)) \setminus F_0$ would be joined by a path (many paths)
having odd intersection with $F_0$, and adding $xy$ to such a path would produce a $C_\kappa$
%, say $Q$, with $|F_0\cap Q|$ odd.
having odd intersection with $F_0$.
(As mentioned earlier, this is the reaon for ``odd" in the definition of central.)
So if $\R$, $\sss$ and $\{F_0\neq \0\}$ hold (and $\vt $ is slightly small) then
\begin{align}
\label{eq:GminusG0capR}
|(G \setminus G_0) \cap R(F_0)| > .12 \alpha n^2p - (1+o(1))\ga n^2q/2  > .1\alpha n^2p.
\end{align}\qedhere

\end{proof}

\begin{proof}[Proof of \eqref{Timplies}]

Set $c=(\kappa-3)/2$.
For $l\in [c]$ and $\0\neq S\sub G_0$
(for $S=\0$ there is nothing to show), call an $xy$-path $(S,l)$-{\em central}
if it is $S$-central and at
least one of its $S$-edges is at distance $l$ (along the path) from one of $x,y$.
(So a path may be $(S,l)$-central for several $l$'s.)
Let $\sigma_0(x,y;S,l)$
be the maximum size of a collection of internally disjoint $(S,l)$-central $xy$-paths
in $G_0$ and
\begin{align}
\label{eq:defRl}
R_l(S) = \{ \{x,y\} \in \Cc{V}{2} : \sigma_0(x,y;S,l) > (.25/c)n^{\kappa-2}q^{\kappa-1}\},
\end{align}
and notice that
\beq{Rifnec}
\mbox{$R(S)\sub \cup_{l\in [c]}R_l(S)$.}
\enq

\medskip
Supposing temporarily (through \eqref{rhophi}) that $S$ and $l$ have been specified, we abbreviate
$\sigma_0(x,y;S,l)=\vs(x,y)$, $R_l(S)=R_l$ and use simply {\em ``rope"} for ``$(S,2l+2)$-rope"
(defined before Lemma~\ref{lem:Claim2}).
Set $|R_l|=\rho_l n^2$ and
\beq{r}
r=2(\kappa-1)-2(l+1) =2(\kappa-l)-4 ~\in [\kappa-1,2\kappa-6].
\enq

\mn
We next show that if $G_0$ satisfies
\beq{T}
\mbox{$T:=\max_{u,v}\tau^r(u,v) = O(n^{r-1}q^r)$}
\enq
(as implied by \eqref{eq:tausummarybig} and
\eqref{eq:tausummarymiddle}, so by $\T$),
then
\beq{rope.count}
\mbox{the number of ropes is
$\Omega(\rho_l n^{2l+3} q^{2l+2} ).$}
\enq

\begin{proof}

Say a rope $P=(u_{l+1}\dots u_1,z,v_1\dots v_{l+1})$
is {\em generated by} $\{x,y\}$
if there are internally disjoint paths
$(z, u_1\dots u_{\kappa-2},w)$ and
$(z,v_1\dots v_{\kappa-2},w)$ with $\{z,w\}=\{x,y\}$.
Each $\{x,y\}\in \binom{V}{2}$ generates at least
$2\binom{\lfloor\vs(x,y)/2\rfloor}{2}$ such ropes
(since a set of $a$
internally disjoint $(S,l)$-central $xy$-paths, each with an $S$-edge
at distance $l$ from $x$, produces $\binom{a}{2}$ of them), while
the number of
pairs generating a given rope is at most $T$
(since in the scenario above, the complement of $P$ in the cycle
$(z,u_1\dots u_{\kappa-2},w,v_{\kappa-2}\dots v_1,z)$
is a path of length $r$ (see \eqref{r})
centered at $w$, so with $P$ determines $\{x,y\}$).
Thus the number of ropes
is at least
\[
\mbox{$T^{-1}\sum_{\{x,y\} \in R_l} 2\binom{\lfloor \vs(x,y)/2\rfloor}{2}
=
\gO(|R_l|(n^{\kappa-2}q^{\kappa-1})^2/T)
= \Omega(\rho_l n^{2l+3} q^{2l+2} ).$}\qedhere
\]
\end{proof}

If we now also assume the conclusion of Lemma~\ref{lem:InterruptedPathsBound}
for $t=2l+2$
(again, this is contained in $\T$), then combining that upper bound
with the lower bound in \eqref{rope.count}
gives
\beq{rhophi}
\rho_l =O(\max\{\ga_S^2, \ga_S (nq)^{-l}\})=O(\ga_S^{1+\gd}),
\enq
with $\delta>0$ depending only on $\kappa$.
(Here we use $\ga_{S}\geq n^{-2}$, valid since $S\neq \0$.)

So, now letting $l$ vary,
it follows that if $G_0$ satisfies $\T$
(and so all relevant instances of \eqref{T} and \eqref{eq:InterruptedPathsBound}),
then \eqref{rhophi} holds for all $l\in [c]$, which in view of \eqref{Rifnec} bounds
$|R(S)|$ as in \eqref{Timplies}.\qedhere

\end{proof}

(It may be worth noting that for $l=0$ the above argument gives only
$\rho_l =O(\ga_S)$, which loses the crucial $\gd$ in \eqref{rhophi}; thus the insistence
on {\em central} paths in $\R$ and Lemma~\ref{lem:Claim2}.)

\begin{proof}[Proof of \eqref{Utoshow}]
Given $G_0$, $S$, we have
$|(G \setminus G_0) \cap R(S)|\sim {\rm Bin}(m,p')$, with $m\leq |R(S)|$ and
$p'<p$
defined by $(1-q)(1-p')=1-p$
(as in (B) of Section~\ref{subsec:Coupling}).
So for $|R(S)|$ as in \eqref{U},
Theorem~\ref{thm:Chernoff'} gives
\[
\Pr(|(G \setminus G_0) \cap R(S)| >  .09\alpha_S n^2p) < \exp[- \gO(\ga_Sn^2p\log (1/\ga_S))],
\]
where the implied constant depends on $\gd$ but not on $\vt$.
Thus, assuming $|G_0|< n^2q$ (as given by $\U$),
setting $\ga_s = 2s/(n^2q) $ (where $s$ will be $|S|$, so $\ga_s=\ga_S$),
and
summing over $s< 2.1\gl n^2q$,
we have
\begin{eqnarray}
\Pr(\U|G_0) &<&\mbox{$\sum_s\binom{n^2q}{s}\exp[- \gO(\ga_sn^2p\log (1/\ga_s))]$}
\nonumber\\
&<&\mbox{$\sum_s \exp[\ga_sn^2p\{(\vt/2) \log (2e/\ga_s)- \gO(\log (1/\ga_s)\}]$},
\label{vt.constraint2}
\end{eqnarray}
which is $o(1)$ for small enough $\vt$
(implying \eqref{Utoshow} since
\[\Pr(\U) =\sum\{\Pr(G_0)\Pr(\U|G_0):|G_0|<n^2q\}).
\]\qedhere

\end{proof}

\subsection{Proof of Lemma~\ref{lem:Claim2}}
\label{subsec:Lemma6.1}

Fix $\eps>0$ (as in \eqref{eq:pgeq1-eps})
small enough to support the proofs of Propositions~\ref{Q's} and \ref{prop:Claim1C}
below; these are our
only constraints on $\eps$, and it will be clear they are satisfiable.
We continue to assume
that $p$ is as in \eqref{eq:pgeq1-eps}.

Most of our effort here is devoted to proving the following
variant of
Proposition~\ref{lem:Claim2} in which we replace $\gs_0(x,y,F_0)$ by
$\gs(x,y,F)$ and $q$ by $p$.

\begin{lemma}
\label{prop:Claim1}
W.h.p.
\begin{align}
\label{eq:Claim1statement}
G \in \Q\wedge \pee \;\; \Longrightarrow \;\;
|\{ xy \in F : \sigma(x,y;F) > .26n^{\kappa-2}p^{\kappa-1}\}| \geq .13\alpha n^2p.
\end{align}
\end{lemma}
\nin
``Coupling down" will then easily get us to Lemma~\ref{lem:Claim2} itself.
(The extra .01's---relative to the pretty arbitrary .25 and .12 in \eqref{R},
\eqref{eq:Claim2statement'}---leave a little room for this.)

\mn
{\em Preview.}
The proof of Lemma~\ref{prop:Claim1} breaks into two parts, roughly (w.h.p.):
(a) if $G \in \Q$ (here we don't need to assume $G\in \pee$),
then $\gs(x,y)$ is close to its natural value for most $xy\in F$
(see the paragraph following the proof of Proposition~\ref{prop:F.99});
(b) a decent fraction of the paths produced in (a) are
$F$-central (shown by limiting the number
that are {\em not}; this is based on Proposition~\ref{prop:Claim1C} and does
assume $G\in\pee$).

\mn
\emph{Definitions.}
%We continue to use the notations $\sigma(x,y)$, $\sigma_0(x,y)$, $\sigma(x,y;S)$, $\sigma_0(x,y;S)$
%defined early in Section \ref{SEC:LemmaP=O(p*)}.
It will be convenient to set
\[
\gL  =n^{\kappa-2}p^{\kappa-1} ,
\]
since this quantity---essentially the typical number of paths in $G$ joining a given pair of
vertices---will appear repeatedly below.
We write $Q\sim Q'$ when $Q, Q'$ are distinct $C_\kappa$'s sharing at least one edge.
%(they will always be contained in $G$).
For edges $e,f$ of $G$, we take
\beq{esimf}
e\sim f ~~\Leftrightarrow ~~ \mbox{[some $C_\kappa$ of $G$ contains both $e$ and $f$]},
\enq
\beq{esimf2}
e\approx f ~~\Leftrightarrow ~~ \mbox{[there are $C_\kappa$'s $Q\sim Q'$ of $G$
with $e\in Q$ and $f\in Q'$]},
\enq
$S(e) =\{g\in G:e\sim g\}$, and $T(e) =\{g\in G:e\approx g\}$.
For $\gamma \in (0,1)$, let
\[
L(\gc) =\{\{x,y\}\in \Cc{V}{2}: \sigma(x,y) < \gamma \gL \}
\] and
$F(\gc)=F\cap L(\gc)$.
Finally, with $C$ as in Proposition~\ref{prop:tausig} for $l=\kappa-1$
(and, say, $\gd=1/\kappa$),
let $\mathcal{S}$ be the event that $G$ satisfies \eqref{eq:tausig}
so \emph{not} the $\sss$ used above).

\medskip
Fix $\gz= .01$.
Our goal in the next four propositions is to show that $F(1-\gz)$ is small,
accomplishing (a) of our outline above.
We do this by showing separately
(in Propositions~\ref{prop:F.01} and \ref{prop:F.99}, using the tools
provided by Propositions~\ref{prop:lightedgesJanson} and \ref{Q's})
that $F(\gz)$
and $F(1-\gz)\sm F(\gz)$ are small.
%Propositions~\ref{prop:lightedgesJanson} and \ref{Q's} are tools for these.

\begin{prop}
\label{prop:lightedgesJanson}
For $\gamma \in (0,1)$ and distinct $\{x_1,y_1\},\ldots,\{x_c,y_c\} \in \binom{V}{2}$,
\begin{align}
\label{eq:lightedgesJanson}
\Pr(\mathcal{S} \wedge  \{\{x_i,y_i\} \in L(\gc) ~\forall \, i \in [c]\} )
\leq n^{-(c-o(1)) (\kappa/(\kappa-1)) (1-\eps)^{\kappa-1} \varphi(\gamma-1)}.
\end{align}
\end{prop}
\nin
(Recall $\varphi(x)$ was defined in \eqref{eq:varphidef}.)
Note the bound here is natural, being, for $p$ at the lower bound in
\eqref{eq:pgeq1-eps} (and up to the $o(1)$), what
Theorem~\ref{thm:Chernoff} would give for the probability that $c$ independent binomials,
each of mean $\gL $,
are all at most $\gamma \gL $.

\begin{proof}
Since $\mathcal{S}$ gives $\tau(x,y) \leq \sigma(x,y)+C <  (1+o(1))\gamma \gL $
for $\{x,y\}\in L(\gamma)$, the event in \eqref{eq:lightedgesJanson} implies
that $X: = \sum_{i \in [c]} \tau(x_i,y_i) < (1+o(1))c\gamma \gL $;
so we just need to bound the probability of this.

%\medskip
In the notation of Theorem \ref{TJanson}, with $A_1\dots A_m$ the edge sets
of the various $x_iy_i$-paths (in $K_n$), we have $\mu\sim c\gL $ and
$\ov{\gD} =\mu+O(\gL ^2/(np))\sim \mu$.
(If two of our paths, say $P$ and $Q$, share $l\in [1,\kappa-2]$ edges, then at least $l$ internal
vertices of $P$ are vertices of $Q$; so the contribution of such pairs to $\ov{\gD}$
is less than
\[
c^2 n^{2(\kappa-2)-l}p^{2(\kappa-1)-l} = O(\gL ^2/(np)) =o(1)
\]
(using the upper bound in \eqref{eq:pgeq1-eps} for the $o(1)$)).
Thus Theorem \ref{TJanson} gives
\[
\Pr(X < (1+o(1))c\gamma \gL ) < \exp\left[-(1-o(1))\varphi(\gamma-1)c\gL \right],
\]
which, since
$\gL  > (1-\eps)^{\kappa-1}(\kappa/(\kappa-1)) \log n$, is less than
the r.h.s. of \eqref{eq:lightedgesJanson}.
\end{proof}

\begin{prop}\label{Q's}
W.h.p.
\beq{Q's1}
\mbox{if
$Q_1 \sim Q_2 \sim Q_3 \sim Q_4$ are $C_\kappa$'s of $G$ then
$|(\cup Q_i)\cap L(\zeta)|\leq 1$}.
\enq
Also, there is a fixed $M$ such that w.h.p.
\beq{Q's2}
|S(e)\cap L(1-\zeta)| < M
~~~\forall \, e \in G.
\enq
\end{prop}
\nin (Note the $Q_i$'s in \eqref{Q's1} need not be distinct.)

\begin{proof}
Write $\eta_\gc$ for the quantity
$
n^{-(1-o(1)) (\kappa/(\kappa-1)) (1-\eps)^{\kappa-1} \varphi(\gamma-1)}
$
appearing in \eqref{eq:lightedgesJanson} (here without the $c$).

Since $\mathcal{S}$ occurs w.h.p.,
it suffices to show that the probability that it holds while either
\eqref{Q's1} or \eqref{Q's2} fails is $o(1)$.
Thus in the case of \eqref{Q's1} we want to bound the probability
that $
%\R(J):=
\mathcal{S}  \wedge  \{J \subseteq G\}  \wedge  \{|J \cap L(\zeta)| \geq 2\}$
holds for some $J\sub K_n$ of the form $\cup_{i \in [4]} Q_i$, where the $Q_i$'s are $C_\kappa$'s
sharing edges as appropriate.  With $\T(J)=\sss\wedge \{|J\cap L(\gz)|\geq 2\}$,
this probability is at most
\begin{align*}
\mbox{$\sum\Pr(\{J\sub G\}\wedge\T(J))$}~
&\leq~
\mbox{$\sum \Pr(J \subseteq G) \Pr(\T(J))~$}\\
&\leq ~\mbox{$O(n^{4\kappa-6}p^{4\kappa-3}\eta_\gz^2) =o(1).$}
\end{align*}
Here the first inequality is an instance of Theorem~\ref{thm:Harris}
(since $\{J \subseteq G\}$ and
$\T(J)$ are increasing and decreasing respectively),
Proposition~\ref{prop:lightedgesJanson} gives
$\Pr(\T(J)) =O(\eta_\gz^2)$ (for any $J$),
and the $o(1)$ holds (for small enough $\eps$)
since $n^{4\kappa-6}p^{4\kappa-3}=\tilde{\Theta}(n^{\kappa/(\kappa-1)})$.
The argument for
\beq{sumJG}
\sum\Pr(J\sub G) = O(n^{4\kappa-6}p^{4\kappa-3})
\enq
is similar to the proof of
Proposition~\ref{prop:tausig}; briefly:
if $Q_1\dots Q_4$ are $C_\kappa$'s, with
$R_i=\cup_{j\leq i}Q_j$ and, for $i\geq 2$,
$|E(Q_i)\sm E(R_{i-1})|=b_i\leq\kappa-1$ and
$|V(Q_i)\sm V(R_{i-1})|=a_i$,
then
$n^{a_i}p^{b_i}\leq \gL$
for $i\geq 2$
(since $b_i=a_i=0$ or $b_i\geq a_i+1$), which gives
$
n^{|V(R_4)|}p^{|E(R_4)|} \leq n^2p\gL^4
$
and \eqref{sumJG}.

\medskip
Treatment of \eqref{Q's2} is similar. Here $J$ runs over subsets of $K_n$
of
the form $\cup_{i \in [M]} Q_i$, where the $Q_i$'s are $C_\kappa$'s with a common edge,
and, with $\T(J) =\sss\wedge \{|J \cap L(1-\zeta)| \geq M\}$,
the probability that $\sss$ holds while \eqref{Q's2} fails is at most
\[
\mbox{$\sum \Pr(\{J \subseteq G\} \wedge \T(J)) $}
~\leq ~O(n^2p \gL^M\eta_{1-\gz}^M) ~=~o(1).
\]
This is shown as above, with
$n^{|V(J)|}p^{|E(J)|} \leq n^2p\gL^M$
given by the passage following \eqref{sumJG}
(with $M$ in place of 4)
and the $o(1)$ valid for large enough $M$ because
$n^2p \gL ^M < n^{\kappa/(\kappa-1)}O(\log^{M/(\kappa-1)}n)$.
\end{proof}

The next assertion is the only place where we use the condition
$\{G \in \Q\}$ of \eqref{eq:Claim2statement'} (and \eqref{eq:ETS1}).

\begin{prop}
\label{prop:F.01}
W.h.p.
\beq{.01}
% \Pr((G \in \Q) \; \wedge \; (|F(\zeta)| > \gd |F|)) \ra 0.
G \in \Q \; \Longrightarrow \; |F(\zeta)| = o( |F|).
\enq
\end{prop}

\begin{proof}
By the first part of Proposition~\ref{Q's} it is enough to show that
%the combination of $\{G \in \Q\}$ and \eqref{Q's1} implies $|F(\zeta)|=o(|F|)$
the r.h.s. of \eqref{.01} follows (deterministically) from
the conjunction of $\{G \in \Q\}$ and \eqref{Q's1}.
But these imply that $|T(e) \cap F| \geq \zeta \gL $ for each $e\in F(\gz)$:
$\{G \in \Q\}$ gives at least one $C_\kappa$ containing $e$;
this $C_\kappa$ contains a second edge, $xy$, of $F$ (since $F\in \cee^\perp_{\kappa}$), which
by \eqref{Q's1} is not in $L(\gz)$; and $T(e)$ contains at least $\zeta \gL $
(distinct) $F$-edges lying
on $xy$-paths.
Moreover, again by \eqref{Q's1},
%$T(e) \cap F(\gz)=\{e\}$ $\forall e\in F(\gz)$ and
$T(e)\cap T(f)=\0 $ for distinct $e,f\in F(\zeta)$.
Thus $|F(\zeta)|< |F|/(\zeta \gL )$ ($=o(|F|)$), as desired.
\end{proof}

\begin{prop}
\label{prop:F.99}
W.h.p.
\beq{.99}
|F(1-\zeta)\sm F(\zeta)| = o(|F|).
\enq
\end{prop}

\begin{proof}
It's enough to show that \eqref{Q's2} implies
\eqref{.99} (since Proposition~\ref{Q's} says \eqref{Q's2} holds w.h.p.).
This is again easy:
Set $B=F(1-\zeta)\sm F(\zeta)$ and consider the graph with vertex set $F$ and
adjacency as in \eqref{esimf}.
Each $e \in B$ has degree at least $\zeta \gL  $ in this graph,
while \eqref{Q's2} says
no vertex has more than $M$ neighbors in $B$.
Thus $|B| (\zeta \gL  - M) \leq |F \sm B|M$, which (since $\gL \gg 1$) gives \eqref{.99}.
\end{proof}

Combining Propositions \ref{prop:F.01} and \ref{prop:F.99}
completes part (a) of the preview following the statement of
Lemma~\ref{prop:Claim1}:
\begin{align}
\label{eq:F.99to3}
\mbox{w.h.p. $~~G \in \Q \; \;\Rightarrow \; \;
|F(1-\gz)| =o(|F|).$}
\end{align}

\mn
The next assertion, an echo of Section~\ref{subsec:Paths},
provides technical support for part (b)
(getting from \eqref{eq:F.99to3} to
Lemma~\ref{prop:Claim1}
by controlling non-$F$-central paths).

\medskip
For $v \in V$ and $S \sub\nabla_G(v)$, let $T_S(v)$
be the set of $C_\kappa$'s using two edges of $S$ and
$\tau_S(v)=|T_S(v)|$.
(We could write simply $T_S, \tau_S$, but keep the $v$ as a reminder).
%Let $h $ be some function of $n$ satifsying $1\ll h\ll \log(e/\gc)$.
%
\begin{prop}
\label{prop:Claim1C}
For each fixed $\theta>0$ there exists $C_\theta$ such that w.h.p.:
%There is a fixed $C$ such that w.h.p.:
for all $v \in V$ and $S \subseteq \nabla_G(v)$, with $|S|=\gc np$ and
$\mu = \gc^2n^{\kappa-1}p^{\kappa}/2$,
\beq{taucases}
\tau_S(v) < \left\{\begin{array}{ll}
(1+\theta)\mu &\mbox{if $\gamma > \gc_{\theta}:= C_\theta \log\log n / \log n$,} \\
o(\mu/\gamma) &\text{in general}.
\end{array}\right.
\enq
\end{prop}

\begin{proof}
We first observe that there is a fixed $B$ such that w.h.p.
no $v$ lies in more than $B$
$C_\kappa$'s that meet $N(v)$ more than twice
(basically because---here we omit the routine details---the
expected number of such $C_\kappa$'s at a given $v$
is $O(n^{\kappa-1}p^{\kappa+1})=n^{-\gO(1)}$).
It is thus enough to prove Proposition~\ref{prop:Claim1C} with $T$ and $\tau$
replaced by $T'$ and $\tau'$, where
$T'_S(v)= \{Q\in T_S(v):|Q\cap N(v)|=2\}$ and $\tau_S'(v) = |T_S'(v)|$.

Here we use a reduction similar to the one given by Proposition~\ref{prop:tausig}
(though 
we can't expect to do quite as well as
in \eqref{eq:tausig}).
Let $\gs_S(v)$ be the maximum size of a collection of
$C_\kappa$'s from $T_S'(v)$ that are disjoint
outside $\ov{N}(v):=\{v\}\cup N(v)$.  Set
$\psi(S)=\min\{|S|, \log ^2n\}$.
\begin{prop}
\label{PlikeP316}
There exists D such that
w.h.p. for all $v$ and $S\sub \nabla_G(v)$, % with $|S|=\gc np$,
\beq{tausig2}
\tau_S'(v)-\gs_S(v) < D\psi(S).
\enq
\end{prop}

\begin{proof}
For fixed $v$ and $S \subseteq \nabla_G(v)$, let $\gG=\gG_S$ be the graph on $T_S'(v)$ with
$Q\sim R$ if $Q$ and $R$ share a vertex not in $\ov{N}(v)$.
Since
$\tau_S'(v)-\gs_S(v)\leq |E(\gG)|$,
\eqref{tausig2} holds (for a suitable $D$) provided
\begin{itemize}
\item[(i)] the sizes of the components of $\gG$ are $O(1)$ and
\item[(ii)] the sizes of the induced matchings of $\gG$ are $O(\psi(S))$;
\end{itemize}
so we would like to say that w.h.p.\ (i) and (ii) hold for all $v$ and $S$.
Here (and only here) we use $V(Q)$ for the set of vertices of $Q$ not in $\ov{N}(v)$.

Of course (i) holds for all $S$ (at $v$) iff it holds for $S=\nabla_G(v)$, so we just
consider this case.
Here we again (as in Proposition~\ref{prop:tausig})
want, for large enough $M$, (probable) nonexistence of
$Q_1\dots Q_M\in T'_S(v)$ such that, for
$i\geq 2$, $V(Q_i)$ meets, but is not contained in,
$\cup_{j<i}V(Q_j)$.
Arguing as for \eqref{nVRM} we find that the total numbers,
say $a$ and $b$, of vertices (other than $v$) and edges used by such
$Q_1\dots Q_M$ satisfy
\beq{napb'}
n^ap^b\leq  n^{\kappa-1}p^{\kappa}(n^{\kappa-3}p^{\kappa-2})^{M-1}.
\enq
(Note here we do count neighbors of and edges at $v$.
The bound says $n^ap^b$ is largest when each new $Q_i$ meets
what preceded it in a $P_2$ starting at $v$.)
Since $n^{\kappa-1}p^{\kappa} =\Theta(np\log n)$ and
$n^{\kappa-3}p^{\kappa-2} =\tilde{\Theta}(n^{-1/(\kappa-1)})$,
the bound in \eqref{napb'} is $o(1/n)$ for slightly large $M$, as is
the probability of seeing such $Q_i$'s at $v$.

For (ii), it will help to condition on $\nabla_G(v)$.
Using $\nu'$ for the maximum size of an induced matching
and invoking Proposition~\ref{prop:routine}, we find that
it's enough to show
that, for a given $v$, $R \sub \nabla(v)$ of size
less than $2np$ (say) and large enough $D$,
\beq{Prtau'v}
\Pr(\exists S\sub R, \nu'(\gG_{S}) >D \psi(S) \; | \; \nabla_G(v)=R) =o(1/n).
\enq
So assume we have conditioned on $\{\nabla_G(v)=R\}$, with $R$ as above.
An easy verification (again similar to those in the proof of Proposition~\ref{prop:tausig})
gives, for any $S\sub R$ and, again, $\gc_S=\gc$ and $\gG_S=\gG$),
\begin{align}
\notag
\tilde{\mu}= \tilde{\mu}_S :=\mathbb{E} |E(\gG)| &=
O(\Cc{|S|}{2}n^{\kappa-3}p^{\kappa-2}|S|n^{\kappa-4}p^{\kappa-3}) \\
&= %O(n^{\kappa-1}p^{\kappa}n^{\kappa-3}p^{\kappa-2}) =
O(\gamma^3 n^{2(\kappa-2)}p^{2(\kappa-1)})
= O(\gamma^3 \log^2n); \label{newlabel}
\end{align}
say $\tilde{\mu} < C\gamma^3 \log^2n$ (with $C$ fixed).
On the other hand,
with $\{Q_i,R_i\}$ the possible edges of $\gG$ and $A_i=\{Q_i\cup R_i\sub G\}$,
$\nu'(\gG)\geq l$ implies occurrence of some $l$ independent $A_i$'s,
an event whose probability \eqref{ET} bounds by
$(e\tilde{\mu}/l)^l$. 

This leaves us with the union bound arithmetic.
Here we first note that for $\nu'(\gG_S) < D\log^2n$ $\forall S$ we just need to check $S=R$,
for which,
in view of \eqref{newlabel}, we have $(e\tilde{\mu}/l)^l =o(1/n)$ for $l=D\log^2n$
with a suitable $D$ ($D> Ce$ is enough).
We then need to say (again, for suitable $D$) that with probability $1-o(1/n)$,
\beq{smallS}
\nu'(\gG_S) < D |S|  ~~\mbox{for all $S$ with $|S|< \log ^2n$.}
\enq
But with $s=\gc np$,
$\tilde{\mu} =\tilde{\mu_s} < C\gc^3\log ^2 n$
and
sums over $s\in [1,\log^2n]$,
the probability that \eqref{smallS} fails is at most
\[
\sum  \Cc{|R|}{s}\left(\tfrac{e\tilde{\mu}}{Ds}\right)^{Ds}
< \sum \exp\left[\gamma np\left\{\log(2e/\gamma) + D \log\left(\tfrac{Ce\gc^3\log^2n}
{D\gc np}\right)\right\}\right],
\]
which, since we are in the range $\gc np \in [1,\log^2n]$,
is easily $o(1/n)$.
%since for $s$ in this range, the argument of the second
%log is essentially $(\gc^2/np)=-\tilde{\gO}(1/\gc)$
\qedhere

\end{proof}

\medskip
We continue with the proof of Proposition~\ref{prop:Claim1C}, which,
by Proposition~\ref{PlikeP316}, we now need only prove with $\tau_S(v)$ replaced by
$\gs_S(v)$.
Here it will help to have a concrete $o(\cdot)$ in \eqref{taucases}.
Set $h=h(n)=(\log\log n)^{1/2}$
(we need $1\ll h\ll \log\log n$) and,
with $C_\theta$ (and thus $\gc_{\theta}$) TBA, set
\[
K_\gc =\left\{\begin{array}{ll}
1+\theta &\mbox{if $\gc >\gc_{\theta}$,}\\
(h\gc)^{-1}&\mbox{otherwise.}
\end{array}\right.
\]

\medskip
Given $v$ and $S\sub \nabla_{G}(v)$ of size $\gc np$
(so we condition on $\{S\sub G\}$), and writing $K$ for $K_\gc$,
we may apply Lemma~\ref{JUB} to obtain
\beq{Xbd}
\Pr(\gs_S(v)> K\mu) < \left\{\begin{array}{ll}
\exp[-\theta^2\mu/3]&\mbox{if $\gc >\gc_{\theta}$,}\\
\exp[- K\mu \log (K/e)] &\mbox{otherwise.}
\end{array}\right.
\enq         %\enq
Thus,
with $\xi_\gc$ denoting the appropriate bound in \eqref{Xbd},
the probability of violating the $\gs_S$-version of \eqref{taucases}
with an $S$ of size $\gc np$ is less than
\beq{gSbd}
n\Cc{n}{\gc np}p^{\gc np}\xi_\gc
< \exp[\log n + \gc np\log (e/\gc)]\cdot\xi_\gc
\enq
(where the terms preceding $\xi_\gc$ correspond to summing $\Pr(S\sub G)$
over $v\in V$ and $S\sub \nabla (v)$
of size $\gc np$).

Finally, we should make sure the bound in \eqref{gSbd} is small.
Recalling \eqref{eq:pgeq1-eps},
we have (for slightly small $\eps$)
$\gL > (1-\eps)^{\kappa-1}\kappa/(\kappa-1)\log n> \log n$ and
\beq{mu}
\mu ~~(= (\gc^2np/2)\gL)~
%> (1-\eps)^{\kappa-1}\kappa/(\kappa-1)\log n
> (\gc^2np/2) \log n.
\enq
Thus for $\gc>\gc_{\theta}$ the bound in \eqref{gSbd} is less than
\[
\exp[\gc np\cdot \{\log (e/\gc) - \theta^2\gc \log n/6\} +\log n],
\]
which is tiny ($\exp[-n^{\gO(1)}]$)
for fixed $C_\theta> 6\theta^{-2}$.

For $\gc\leq \gc_{\theta}$, noting that
$(\gc K_{\gc}/2)\log (K_{\gc}/e)\sim \log(1/\gc)/(2h) =\go(1)$
(and $\gc np\geq 1$), and again using \eqref{mu},
we find that
the r.h.s. of \eqref{gSbd} is less than
\[
\exp[\gc np \cdot\{\log (e/\gc) - (\gc K_{\gc}/2)\log (K_{\gc}/e) \log n\} +\log n]
~=~ n^{-\go(1)}.
\]
And of course summing these bounds over $\gc$ gives what we want.\qedhere

\end{proof}

\begin{proof}[Proof of Lemma~\ref{prop:Claim1}.] Fix $\theta = .005$
and let $C=C_\theta$ and $\gc_{\theta}$ be as in Proposition~\ref{prop:Claim1C}.
Set $\gamma_v = d_F(v)/(np)$, and let $\varphi_v$ be the number of $C_\kappa$'s of $G$ using
two $F$-edges at $v$.
Let $\sigma^*(x,y)$ be the number of $xy$-paths having $F$-edges at one or both of $x,y$.
Write $\sum'$ and $\sum''$ for sums over $v$ with $\gc_v>\gc_{\theta}$ 
and $\gc_v\leq\gc_{\theta}$ respectively.
We have, w.h.p.,
\begin{align}
\mbox{$\sum_{xy \in F} \sigma^*(x,y)$} &\leq \mbox{$2\sum_{v \in V} \varphi_v $}\nonumber\\
\label{eq:Claim1eq2}
&\leq \mbox{$n^{\kappa-1}p^{\kappa} \cdot [(1+\theta)\sum'\gamma_v^2
+ \sum''o(\gamma_v)],$}
\end{align}
where the first inequality comes from considering how many times each side counts the various
$C_\kappa$'s of $G$,
and the second is given by Proposition \ref{prop:Claim1C}.

Since $\sum \gamma_v = \alpha n$, the second sum in \eqref{eq:Claim1eq2} is $o(\ga n)$.
For the first,
let $B=\{ v \in V : \gamma_v > \theta\}$.
If we now assume
$\alpha = o(1)$ (as given by $\pee$), then we have $|B|=o(n)$; so
Proposition~\ref{density} (parts (a) and (b)) gives (w.h.p.)
\[
\mbox{$|G[B]| \ll |B|\theta np
<
\sum_{v \in B} d_F(v)  \leq \ga n^2p$},
\]
whence
$
\mbox{$\sum_{v \in B} \gamma_v np \leq 2|G[B]| +|\nabla_F(B)|
< (1+o(1))\ga n^2p/2$,}
$
\[
\mbox{$ \sum_{v \in B} \gamma_v<(1+o(1))\alpha n/2$}
\]
and
(recalling $d_F(v)\leq d_G(v)/2 ~\forall v$; see \eqref{dFdG})
\beq{ga/4}
\mbox{$ \sum_{v \in B} \gamma_v^2 \leq\max_v \gc_v \sum_{v\in B}\gc_v
< (1+o(1))\ga n/4.$}
\enq
Thus (since also $\sum_{v \in V \setminus B} \gamma_v^2\leq \theta \sum_v\gc_v =\theta\ga n$)
we find that the expression in square brackets in \eqref{eq:Claim1eq2} is less than
$ (1/4 + 2\theta )\alpha n,$
whence
\begin{align}
\label{eq:Claim1eq6}
\mbox{$\sum_{xy \in F} \sigma^*(x,y) \leq (1/4 + 2\theta )\ga n^{\kappa}p^{\kappa}
=.26\alpha n^\kappa p^\kappa.$}
\end{align}
(To avoid confusion we note that
the .26 here, which is more or less forced by the essentially tight bound
in \eqref{ga/4}, is unrelated to the one in \eqref{eq:Claim1statement}.)

\medskip
Now let $F^* = \{xy \in F : \sigma(x,y) \geq (1-\gz)\gL \}$
 ($= F\sm F(1-\gz)$).
By
\eqref{eq:F.99to3}, $|F^*| \sim \alpha n^2p/2$, w.h.p.\ provided $\Q$ holds.
Note that (recall $\gz =.01$)
$xy\in F^*$ has
$\sigma(x,y;F) > .26\gL$ (as in \eqref{eq:Claim1statement})
unless $\sigma^*(x,y) > .73\gL$.
(As noted earlier, $xy$-paths necessarily have odd intersection with $F$,
so the only real requirement for such a path to be central is that it have an internal edge in $F$.)
It thus follows from \eqref{eq:Claim1eq6} that for
$\tilde{F} := \{xy \in F^* :\sigma(x,y;F) \leq .26\gL\}$, we have
\[
|\tilde{F}| \leq \frac{.26\alpha n^\kappa p^\kappa}{.73\gL} \leq .36\alpha n^2p,
\]
whence $|F^* \setminus \tilde{F}| \geq .13\alpha n^2p$, implying \eqref{eq:Claim1statement}.
\end{proof}

\begin{proof}[Proof of Lemma~\ref{lem:Claim2}]
As mentioned earlier, Lemma~\ref{lem:Claim2} follows easily from Lemma~\ref{prop:Claim1}
{\em via} ``coupling down" (viewpoint (A) of Section~\ref{subsec:Coupling}):
it is enough to show that if $G$ satisfies the r.h.s. of \eqref{eq:Claim1statement} then
w.h.p.\ it also satisfies $\R$; that is,
$|F \cap R(F_0)| \geq .12\alpha n^2p$.

\medskip
For
$xy\in F':= \{xy\in F:\sigma(x,y;F) > .26\gL\}$
(see \eqref{eq:Claim1statement}),
Theorem~\ref{thm:Chernoff} gives
\[
\Pr(\sigma_0(x,y;F_0) \leq .25n^{\kappa-2}q^{\kappa-1}) < \exp[-\Omega(n^{\kappa-2}q^{\kappa-1})]
= n^{-\gO(1)},
\]
since members of a set of $\sigma(x,y;F)$ internally disjoint, $F$-central
$xy$-paths survive in $G_0$ (and become $F_0$-central) independently, each with probability
$\vt^{\kappa-1}$.
So by Markov's Inequality,
w.h.p.
\[
|\{xy\in F': \sigma_0(x,y;F_0) \leq .25n^{\kappa-2}q^{\kappa-1}\}|=o(|F'|).
\]
The lemma follows.
\end{proof}

\subsection{Proof of Lemma \ref{lem:InterruptedPathsBound}}
\label{subsec:Lemma6.2}

This is a simple consequence of Proposition~\ref{prop:spectrum}, but
for perspective a brief comment on the bounds may helpful.
The first bound---corresponding to a $\gb^2$-fraction of all
$P_t$'s having their ends in $S$---is the generic value,
and will be the truth if $\q$ is large enough that (w.h.p.) all $\tau^{t-2}(x,y)$'s
are about the same.  For smaller $\q$ one can sometimes do better
by, e.g.\ (for even $t$), taking $S$ to consist of all edges at distance $t/2-1$ from
some small set of ``centers," producing something like the second bound.

\begin{proof}
Let $\lambda_1 \geq \lambda_2 \geq \ldots \geq \lambda_n$ be the eigenvalues of the
adjacency matrix, $A$, of $G_0$, with associated
orthonormal eigenvectors $v_1, v_2, \ldots, v_n$, say with $\max_j v_{1,j} > 0$. Let $M=A^{t-2}$
(so $M$ has eigenvalues $\gl_i^{t-2}$ ($i\in [n]$), with eigenvectors $v_i$),
and $f=( d_S(x) : x\in V) =\sum\gb_iv_i$.

The number of $(S,t)$-ropes may w.h.p. be bounded by
\begin{align}
fMf^T &= \sum \gl_i^{t-2}\gb_i^2  \notag \\
&\leq \gl_1^{t-2} \gb_1^2 + \max\{\gl_2,|\gl_n|\}^{t-2} \|f\|_2^2 \notag \\
\label{eq:Stropecount}
&< (1+o(1))[(n\q)^{t-2} \gb_1^2 + (4n\q)^{(t-2)/2} \|f\|_2^2],
% &= & (1+o(1))a^2n^{t+1}\q^t + O(a n^{t/2+2} \q^{t/2+1}).
\end{align}
where we used $\sum \gb_i^2=\|f\|_2^2$ and
the second inequality is given by \eqref{eq:eigvals}.
We then need bounds on $\beta_1^2$ and $\|f\|_2^2$,
both of which are easy:  w.h.p.
\[
\mbox{$\gb_1=\langle f,v_1\rangle \sim n^{-1/2} \sum d_S(v) = 2n^{-1/2}|S|
= \gb n^{3/2}\q$}
\]
(using \eqref{eq:eigvecv1}) and
\[
\|f\|_2^2 = \sum d_S^2(x) \leq \gD_S \sum_xd_S(x)
< (1+o(1))nq\cdot 2|S|
 \sim \gb n^3\q^2.
\]
The lemma follows.\end{proof}

\section{Proof of Theorem~\ref{thm:Hspace}}
\label{SEC:ThmHspace}

In what follows we set $\eee(K_n)=\eee$, $\cee_H(K_n)=\cee_H$ and so on.
We prove (sketchily) Theorem~\ref{thm:Hspace} for $n\geq v_H+2$---which is best possible
e.g.\ if $H=K_\kappa$ with $\kappa\geq 4$ (e.g. since for $n\leq \kappa +1$,
$\cee_H^\perp\supseteq \cee\cap \D$)---and
add a note at the end to cover $H=C_\kappa$ and $n\geq \kappa$.

We first note that $\cee_H =\eee$ if $|H|=1$ (trivially) and
$\cee_H =\D$ if $|H|=2$.
(Since each of $P_2$, $2K_2$ (a 2-edge matching) is the sum of two copies of the other,
the copies of an $H$ of size 2 span all 2-edge subgraphs, and so all even subgraphs, of $K_n$.)
Moreover, if $H$ is a matching then $\cee_H$ is easily seen to contain (all copies of)
$K_2$ if $|H|$ is odd or $2K_2$ if $|H|$ is even, so is equal to $\eee$ or $\D$ as appropriate.

We may thus restrict attention to $H$ containing a vertex $x$  of degree at least 2,
and observe that in this case $\cee_H\supseteq\cee_4$.
(The sum of two copies of $H$ that differ only in the copy of $x$ is a $K:=K_{2,d(x)}$,
and repeating this with $K$ and one of its divalent vertices produces a $C_4$.)

Since $\cee_4 = \cee \cap \D$, we're done if $H$ is even Eulerian.
Otherwise let $\tilde{H}$ be a copy of $H$ in $K_n$ and $F$ a smallest element 
of $\tilde{H}+\cee_4$. Then $F$ clearly belongs to the same case of \eqref{eq:Hspace} as $H$ and 
we claim it is either a triangle or 
the disjoint union of a matching and star (so possibly just a matching or just a star). 
Note this is enough, as the copies of 
$F$ are then easily seen to generate the desired subspace of $\eee$: 
if $H$ is Eulerian then $F=K_3$; 
otherwise we may add two copies of $F$ to produce a $P_2$, so the generated space contains $\D$.  
(Minor note:  $|V(F)|\leq |V(H)|+1$ since all odd vertices of $F$
must also be odd in $\tilde{H}$.)

For the claim we observe that $F$ cannot contain
a $P_3$ 
(since adding a $C_4$ containing such a $P_3$ reduces $|F|$); 
disjoint $P_2$'s (reduce by adding a $C_6$); or $K_3+K_2$ (convert to $P_4$, then
reduce to $P_2$).

\mn

Finally, for $H=C_\kappa$ and $n\geq \kappa\geq 4$
(for $\kappa =3$ there is nothing to show),
it is enough to observe that the sum of two copies of $H$ on the same vertex set
and sharing a $P_{\kappa -3}$ is a $C_4$;
so $\cee_H=\cee\cap \D$ if $\kappa$ is even, while for odd $\kappa$,
$\cee\cap \D\subset \cee_H\sub \cee$ implies $\cee_H= \cee$.

% BIBLIOGRAPHY

% https://verbosus.com/bibtex-style-examples.html -- EXAMPLE BIBTEX ENTRIES, or see
% https://en.wikibooks.org/wiki/LaTeX/Bibliography_Management#Standard_templates
% to compile: pdflatex bibtex pdflatex pdflatex

% \nocite{*}  % to include all sources in References list--DO NOT DO THIS IN FINAL VERSION
\bibliographystyle{plain}
\bibliography{Allrefs.bib}

\iffalse

\fi

\end{document}